\documentclass[reqno,oneside]{amsart}
\usepackage[foot]{amsaddr}
\usepackage{amsmath}
\usepackage{amsthm}
\usepackage{amssymb}
\usepackage{bbm}
\usepackage{bm}
\usepackage{stmaryrd}
\usepackage{tikz}
\usepackage{mathtools}
\counterwithin{equation}{section}
\usepackage{mathrsfs}

\usepackage{pgfplots}
\pgfplotsset{compat=1.16}
\usepackage{subcaption}
\usepackage{float}
\usepackage{graphicx}
\usepackage{wrapfig}
\usepackage{caption}
\usepackage[bottom]{footmisc}
\usepackage{url}
\usepackage[utf8]{inputenc}
\usepackage[T1]{fontenc}
\usepackage{enumitem}
\usepackage[normalem]{ulem}
\usepackage{booktabs}
\usepackage{algorithm}
\usepackage[noend]{algpseudocode}
\algrenewcommand\algorithmicindent{1.5em}
\algrenewcommand\alglinenumber[1]{\footnotesize #1.\enspace}
\usepackage[numbers,sort]{natbib}
\usepackage{todonotes}

\usepackage[hidelinks]{hyperref}
\usepackage[stretch=10]{microtype}
\usepackage{cleveref}
\crefformat{equation}{(#2#1#3)}
\newcommand*{\doi}[1]{\href{http://doi.org/#1}{\texttt{doi:#1}}}
\newcommand*{\customdoi}[2]{\href{http://doi.org/#1}{\texttt{#2}}}
\algdef{SE}[SUBALG]{Indent}{EndIndent}{}{\algorithmicend}%
\algtext*{Indent}
\algtext*{EndIndent}

\algnewcommand{\LineComment}[1]{\State \texttt{/* #1 */}}
\def\algbackskip{\hskip-\ALG@thistlm}
\calclayout

\usepackage[framemethod=TikZ]{mdframed}
\usepackage{xcolor}
\usepackage{siunitx}

\theoremstyle{plain}
\newtheorem{theorem}{Theorem}[section]

\newtheorem{corollary}[theorem]{Corollary}

\newtheorem{lemma}[theorem]{Lemma}

\theoremstyle{definition}
\newtheorem{definition}{Definition}
\newtheorem{remark}[theorem]{Remark}
\newtheorem*{remark*}{Remark}

\newcommand{\eg}{{e.g.,}\ }

\newcommand{\nref}[2]{\hyperref[#1]{#2}}

\newcommand{\cadlag}{c\`adl\`ag{}}

\newcommand{\bc}[1]{{\text{\fontsize{5}{5}\selectfont{$($}{\hspace*{-0.3pt}}\fontsize{5}{4}\selectfont$#1$\hspace*{-0.3pt}\fontsize{5}{5}\selectfont{$)$}}}}

\newcommand{\de}{\vcentcolon=}

\newcommand{\dgeo}{d_{\rm{geo}}}
\renewcommand{\inf}{\mathop{\mathrm{inf}\vphantom{\mathrm{sup}}}}

\newcommand{\intd}{{\rm d}}

\newcommand{\Area}{\operatorname{Area}}

\newcommand{\tmix}{t_{\rm{mix}}}
\newcommand{\Unif}[1]{\operatorname{Unif}#1}

\newcommand{\defbold} [1]{\expandafter\newcommand\csname b#1\endcsname{\mathbf{#1}}}
\newcommand{\defcal} [1]{\expandafter\newcommand\csname c#1\endcsname{\mathcal{#1}}}
\newcommand{\defbb} [1]{\expandafter\newcommand\csname bb#1\endcsname{\mathbb{#1}}}
\newcommand{\defscr} [1]{\expandafter\newcommand\csname s#1\endcsname{\mathscr{#1}}}
\newcommand{\deffrak} [1]{\expandafter\newcommand\csname f#1\endcsname{\mathfrak{#1}}}
\newcommand{\defbracketed} [1]{\expandafter\newcommand\csname bc#1\endcsname{\bc{#1}}}
\newcommand{\deftilde} [1]{\expandafter\newcommand\csname ti#1\endcsname{\tilde{#1}}}
\newcommand{\defoverline} [1]{\expandafter\newcommand\csname ol#1\endcsname{\overline{#1}}}
\newcommand{\defvec} [1]{\expandafter\newcommand\csname vec#1\endcsname{\vec{#1}}}

\expandafter\defbold{A}
\expandafter\defbold{B}
\expandafter\defbold{C}
\expandafter\defbold{D}
\expandafter\defbold{E}
\expandafter\defbold{F}
\expandafter\defbold{G}
\expandafter\defbold{H}
\expandafter\defbold{I}
\expandafter\defbold{J}
\expandafter\defbold{K}
\expandafter\defbold{L}
\expandafter\defbold{M}
\expandafter\defbold{N}
\expandafter\defbold{O}
\expandafter\defbold{P}
\expandafter\defbold{Q}
\expandafter\defbold{R}
\expandafter\defbold{S}
\expandafter\defbold{T}
\expandafter\defbold{U}
\expandafter\defbold{V}
\expandafter\defbold{W}
\expandafter\defbold{X}
\expandafter\defbold{Y}
\expandafter\defbold{Z}

\expandafter\defcal{A}
\expandafter\defcal{B}
\expandafter\defcal{C}
\expandafter\defcal{D}
\expandafter\defcal{E}
\expandafter\defcal{F}
\expandafter\defcal{G}
\expandafter\defcal{H}
\expandafter\defcal{I}
\expandafter\defcal{J}
\expandafter\defcal{K}
\expandafter\defcal{L}
\expandafter\defcal{M}
\expandafter\defcal{N}
\expandafter\defcal{O}
\expandafter\defcal{P}
\expandafter\defcal{Q}
\expandafter\defcal{R}
\expandafter\defcal{S}
\expandafter\defcal{T}
\expandafter\defcal{U}
\expandafter\defcal{V}
\expandafter\defcal{W}
\expandafter\defcal{X}
\expandafter\defcal{Y}
\expandafter\defcal{Z}

\expandafter\defbb{A}
\expandafter\defbb{B}
\expandafter\defbb{C}
\expandafter\defbb{D}
\expandafter\defbb{E}
\expandafter\defbb{F}
\expandafter\defbb{G}
\expandafter\defbb{H}
\expandafter\defbb{I}
\expandafter\defbb{J}
\expandafter\defbb{K}
\expandafter\defbb{L}
\expandafter\defbb{M}
\expandafter\defbb{N}
\expandafter\defbb{O}
\expandafter\defbb{P}
\expandafter\defbb{Q}
\expandafter\defbb{R}
\expandafter\defbb{S}
\expandafter\defbb{T}
\expandafter\defbb{U}
\expandafter\defbb{V}
\expandafter\defbb{W}
\expandafter\defbb{X}
\expandafter\defbb{Y}
\expandafter\defbb{Z}

\expandafter\defscr{A}
\expandafter\defscr{B}
\expandafter\defscr{C}
\expandafter\defscr{D}
\expandafter\defscr{E}
\expandafter\defscr{F}
\expandafter\defscr{G}
\expandafter\defscr{H}
\expandafter\defscr{I}
\expandafter\defscr{J}
\expandafter\defscr{K}
\expandafter\defscr{L}
\expandafter\defscr{M}
\expandafter\defscr{N}
\expandafter\defscr{O}
\expandafter\defscr{P}
\expandafter\defscr{Q}
\expandafter\defscr{R}
\expandafter\defscr{S}
\expandafter\defscr{T}
\expandafter\defscr{U}
\expandafter\defscr{V}
\expandafter\defscr{W}
\expandafter\defscr{X}
\expandafter\defscr{Y}
\expandafter\defscr{Z}

\expandafter\deffrak{A}
\expandafter\deffrak{B}
\expandafter\deffrak{C}
\expandafter\deffrak{D}
\expandafter\deffrak{E}
\expandafter\deffrak{F}
\expandafter\deffrak{G}
\expandafter\deffrak{H}
\expandafter\deffrak{I}
\expandafter\deffrak{J}
\expandafter\deffrak{K}
\expandafter\deffrak{L}
\expandafter\deffrak{M}
\expandafter\deffrak{N}
\expandafter\deffrak{O}
\expandafter\deffrak{P}
\expandafter\deffrak{Q}
\expandafter\deffrak{R}
\expandafter\deffrak{S}
\expandafter\deffrak{T}
\expandafter\deffrak{U}
\expandafter\deffrak{V}
\expandafter\deffrak{W}
\expandafter\deffrak{X}
\expandafter\deffrak{Y}
\expandafter\deffrak{Z}

\expandafter\defbracketed{A}
\expandafter\defbracketed{B}
\expandafter\defbracketed{C}
\expandafter\defbracketed{D}
\expandafter\defbracketed{E}
\expandafter\defbracketed{F}
\expandafter\defbracketed{G}
\expandafter\defbracketed{H}
\expandafter\defbracketed{I}
\expandafter\defbracketed{J}
\expandafter\defbracketed{K}
\expandafter\defbracketed{L}
\expandafter\defbracketed{M}
\expandafter\defbracketed{N}
\expandafter\defbracketed{O}
\expandafter\defbracketed{P}
\expandafter\defbracketed{Q}
\expandafter\defbracketed{R}
\expandafter\defbracketed{S}
\expandafter\defbracketed{T}
\expandafter\defbracketed{U}
\expandafter\defbracketed{V}
\expandafter\defbracketed{W}
\expandafter\defbracketed{X}
\expandafter\defbracketed{Y}
\expandafter\defbracketed{Z}

\expandafter\deftilde{A}
\expandafter\deftilde{B}
\expandafter\deftilde{C}
\expandafter\deftilde{D}
\expandafter\deftilde{E}
\expandafter\deftilde{F}
\expandafter\deftilde{G}
\expandafter\deftilde{H}
\expandafter\deftilde{I}
\expandafter\deftilde{J}
\expandafter\deftilde{K}
\expandafter\deftilde{L}
\expandafter\deftilde{M}
\expandafter\deftilde{N}
\expandafter\deftilde{O}
\expandafter\deftilde{P}
\expandafter\deftilde{Q}
\expandafter\deftilde{R}
\expandafter\deftilde{S}
\expandafter\deftilde{T}
\expandafter\deftilde{U}
\expandafter\deftilde{V}
\expandafter\deftilde{W}
\expandafter\deftilde{X}
\expandafter\deftilde{Y}
\expandafter\deftilde{Z}

\expandafter\defoverline{A}
\expandafter\defoverline{B}
\expandafter\defoverline{C}
\expandafter\defoverline{D}
\expandafter\defoverline{E}
\expandafter\defoverline{F}
\expandafter\defoverline{G}
\expandafter\defoverline{H}
\expandafter\defoverline{I}
\expandafter\defoverline{J}
\expandafter\defoverline{K}
\expandafter\defoverline{L}
\expandafter\defoverline{M}
\expandafter\defoverline{N}
\expandafter\defoverline{O}
\expandafter\defoverline{P}
\expandafter\defoverline{Q}
\expandafter\defoverline{R}
\expandafter\defoverline{S}
\expandafter\defoverline{T}
\expandafter\defoverline{U}
\expandafter\defoverline{V}
\expandafter\defoverline{W}
\expandafter\defoverline{X}
\expandafter\defoverline{Y}
\expandafter\defoverline{Z}

\expandafter\defvec{A}
\expandafter\defvec{B}
\expandafter\defvec{C}
\expandafter\defvec{D}
\expandafter\defvec{E}
\expandafter\defvec{F}
\expandafter\defvec{G}
\expandafter\defvec{H}
\expandafter\defvec{I}
\expandafter\defvec{J}
\expandafter\defvec{K}
\expandafter\defvec{L}
\expandafter\defvec{M}
\expandafter\defvec{N}
\expandafter\defvec{O}
\expandafter\defvec{P}
\expandafter\defvec{Q}
\expandafter\defvec{R}
\expandafter\defvec{S}
\expandafter\defvec{T}
\expandafter\defvec{U}
\expandafter\defvec{V}
\expandafter\defvec{W}
\expandafter\defvec{X}
\expandafter\defvec{Y}
\expandafter\defvec{Z}

\expandafter\defbracketed{a}
\expandafter\defbracketed{b}
\expandafter\defbracketed{c}
\expandafter\defbracketed{d}
\expandafter\defbracketed{e}
\expandafter\defbracketed{f}
\expandafter\defbracketed{g}
\expandafter\defbracketed{h}
\expandafter\defbracketed{i}
\expandafter\defbracketed{j}
\expandafter\defbracketed{k}
\expandafter\defbracketed{l}
\expandafter\defbracketed{m}
\expandafter\defbracketed{n}
\expandafter\defbracketed{o}
\expandafter\defbracketed{p}
\expandafter\defbracketed{q}
\expandafter\defbracketed{r}
\expandafter\defbracketed{s}
\expandafter\defbracketed{t}
\expandafter\defbracketed{u}
\expandafter\defbracketed{v}
\expandafter\defbracketed{w}
\expandafter\defbracketed{x}
\expandafter\defbracketed{y}
\expandafter\defbracketed{z}

\expandafter\deftilde{a}
\expandafter\deftilde{b}
\expandafter\deftilde{c}
\expandafter\deftilde{d}
\expandafter\deftilde{e}
\expandafter\deftilde{f}
\expandafter\deftilde{g}
\expandafter\deftilde{h}
\expandafter\deftilde{i}
\expandafter\deftilde{j}
\expandafter\deftilde{k}
\expandafter\deftilde{l}
\expandafter\deftilde{m}
\expandafter\deftilde{n}
\expandafter\deftilde{o}
\expandafter\deftilde{p}
\expandafter\deftilde{q}
\expandafter\deftilde{r}
\expandafter\deftilde{s}
\expandafter\deftilde{t}
\expandafter\deftilde{u}
\expandafter\deftilde{v}
\expandafter\deftilde{w}
\expandafter\deftilde{x}
\expandafter\deftilde{y}
\expandafter\deftilde{z}

\expandafter\defoverline{a}
\expandafter\defoverline{b}
\expandafter\defoverline{c}
\expandafter\defoverline{d}
\expandafter\defoverline{e}
\expandafter\defoverline{f}
\expandafter\defoverline{g}
\expandafter\defoverline{h}
\expandafter\defoverline{i}
\expandafter\defoverline{j}
\expandafter\defoverline{k}
\expandafter\defoverline{l}
\expandafter\defoverline{m}
\expandafter\defoverline{n}
\expandafter\defoverline{o}
\expandafter\defoverline{p}
\expandafter\defoverline{q}
\expandafter\defoverline{r}
\expandafter\defoverline{s}
\expandafter\defoverline{t}
\expandafter\defoverline{u}
\expandafter\defoverline{v}
\expandafter\defoverline{w}
\expandafter\defoverline{x}
\expandafter\defoverline{y}
\expandafter\defoverline{z}

\expandafter\defvec{a}
\expandafter\defvec{b}
\expandafter\defvec{c}
\expandafter\defvec{d}
\expandafter\defvec{e}
\expandafter\defvec{f}
\expandafter\defvec{g}
\expandafter\defvec{h}
\expandafter\defvec{i}
\expandafter\defvec{j}
\expandafter\defvec{k}
\expandafter\defvec{l}
\expandafter\defvec{m}
\expandafter\defvec{n}
\expandafter\defvec{o}
\expandafter\defvec{p}
\expandafter\defvec{q}
\expandafter\defvec{r}
\expandafter\defvec{s}
\expandafter\defvec{t}
\expandafter\defvec{u}
\expandafter\defvec{v}
\expandafter\defvec{w}
\expandafter\defvec{x}
\expandafter\defvec{y}
\expandafter\defvec{z}

\def\b1{\mathbf{1}}
\def\bb1{\mathbbm{1}}

\newcommand{\Lc}[1]{L^{\mkern-2mu \bc{#1}}}

\newcommand{\hatn}{\hat{n}}

\newcommand{\nb}{m}

\newcommand{\papertitle}{Worst-case mixing estimates for Brownian motion with semipermeable barriers}
\title{Worst-case mixing estimates for Brownian motion with semipermeable barriers}

\author{Alexander Van Werde} 
\author{Jaron Sanders}
\hypersetup{
    pdftitle={\papertitle},
    pdfauthor={Alexander Van Werde, Jaron Sanders},
    pdfsubject={Brownian motion, mixing times},
    pdfkeywords={Snapping-out Brownian motion,
    Total variation mixing time,
    Doeblin minorization,
    Stationary distribution,
    Reflection}
}

\begin{document}
\begin{abstract}
We study the mixing properties of a Brownian motion whose movements are hindered by semipermeable barriers.
Our setting assumes that the process takes values in a smooth planar domain and that the barriers are one-dimensional closed curves.
We establish an upper bound on the mixing time and a lower bound on the stationary distribution in terms of geometric parameters.
These worst-case bounds decay at an exponential rate as the domain grows large, and we give examples that show that exponential decay is necessary in our worst-case setting.
\end{abstract}
\maketitle

\section{Introduction}

Reflected Brownian motion with semipermeable barriers, or \emph{snapping-out Brownian motion}, is a continuous-time stochastic process whose movements are hindered by semipermeable barriers.
The movements of the process are like those of classical Brownian motion when it is not in contact with a barrier, and in fact, its reflections upon hitting a barrier are also typically like those of standard reflected Brownian motion.
The difference, however, is that it may pass through the barriers at certain random times.
These random times are governed by how much local time the process has accumulated at a barrier, and an on--off Markov chain associated with the barrier; see Definition \ref{def: ReflectedBrownianMotion}.

This process was independently introduced to the literature by various authors starting from 2012. 
It arose in work of N\'andori and Sz\'asz \cite{nandori2012lorentz} as a scaling limit of a Lorentz process with a punctured wall, while 
Lejay \cite{lejay2016snapping} as well as Mandrekar and Pilipenko \cite{mandrekar2016brownian} found it as the limit of processes with a membrane of shrinking thickness, and Bobrowski \cite{bobrowski2012diffusions} studied a variant that lives on the edges of a graph.  
Since then, the process has been used to study a range of interface phenomena \cite{forien2019gene,erhard2021slow,zhao2024voter,slkezak2021diffusion,bressloff2022probabilistic,bressloff2023renewal,bobrowski2017averaging,bobrowski2020modeling,bobrowski2025approximation,bobrowski2025snapping}.
The introduction of the process in \cite{lejay2016snapping} was motivated by the fact that it can be used as a probabilistic realization of a (deterministic) differential equation with interface condition.
Their study led to Monte Carlo estimators for problems arising in applications like diffusion MRI imaging  \cite{lejay2018monte,schumm2023numerical}.
Another application occurs in a statistical question where the goal is to uncover the location of barriers that hinder a process based on observations of this process \cite{vanwerde2024recovering}.
This may be relevant for animal movements or subcellular diffusions.

For snapping-out Brownian motion, it is natural to wonder how the barriers affect the stochastic process' behavior.
In this paper, we investigate specifically how its mixing time and stationary distribution are affected.
Studying these two quantities is not just needed to fundamentally understand the process.
Performance guarantees for, for example, Monte Carlo methods \cite{roberts2004general,paulin2015concentration} such as the aforementioned ones from \cite{lejay2018monte,schumm2023numerical} and statistical estimators that recover barriers from a finite trajectory \cite[Theorem 2.3]{vanwerde2024recovering}, are a function of the mixing time and stationary distribution.

\newpage 
Our results on the mixing time and stationary distribution are qualitative.
Specifically, we establish a generic bound on the mixing time and stationary distribution in terms of explicit geometric parameters given in Section \ref{sec: Parameters}.
We find that this bound decays at an exponential rate when the domain grows large and all other parameters are kept fixed.
By constructing concrete examples, we show that such exponential decay can not be avoided in a worst case; see Section \ref{sec: ResultSub}.

The proof relies on a reduction to the Doeblin minorization condition.
We first show that one can always transport probability mass between nearby regions of the domain, regardless of whether a semipermeable barrier is in the way.
We subsequently use a sequence of such local steps to globalize the estimates.
Visualizations for the local transport steps may be found in Figures \ref{fig: PillShaped} to \ref{fig: AvoidBarrier}, and the method to globalize the estimates with a sequence of steps that follow a geodesic in the domain is visualized in Figure \ref{fig: Geo}.

In Section \ref{sec: RelatedWork}, we next outline related work in the literature.
Focus thus far looks to have primarily been on the mixing properties of canonical (reflected) Brownian motions.
In Section \ref{sec: FutureWork}, we discuss directions for future work.
Finally, our results can be found in Section \ref{sec: Results}, and the proofs in Section \ref{sec: ProofMainMixing}.

\subsection{Related work on mixing properties of Brownian motion}
\label{sec: RelatedWork}

Currently, there do not appear to be other studies of the mixing properties of a snapping-out Brownian motion with semipermeable barriers.
There is, however, such research for classical reflected Brownian motion (so for when there are only impermeable barriers).
To contextualize our work, we will discuss some of this research.

Regarding the mixing time for reflected Brownian motion, one can distinguish the following lines of research:

\begin{enumerate}[leftmargin = 2.3em, label = (\alph*)]
	\item
	      Domains with weak regularity, such as a fractal boundary, were considered by Burdzy, Chen, and Marshall \cite{burdzy2006traps}.
	      Note that in such scenarios it is, for example, not immediately clear whether the process is uniformly ergodic.
	      \cite{burdzy2006traps} studied geometric necessary and sufficient conditions for \emph{trap domains} where the process is not uniformly ergodic.
	\item
	      For compact convex domains, the process will certainly be ergodic.
	      One can therefore always study the rate of mixing in such scenarios \cite{loper2020uniform,matthews1990mixing}.
	      For instance, Loper \cite{loper2020uniform} established bounds on the mixing time in this setting that only depend on the diameter of the set and not on the dimensionality.
	\item
	      Mixing rates for the orthant with drift towards the origin were considered by Glynn and Wang \cite{glynn2018rate} for the total variation distance, and by Blanchet and Chen \cite{blanchet2016rates} as well as Banerjee and Budhiraja \cite{banerjee2020parameter} for the Wasserstein distance; see also the references therein.
\end{enumerate}

A significant motivation for the study of the stationary distribution of reflected Brownian motion comes from the heavy-traffic approximation of queuing systems.
One is then typically interested in a polyhedral domain, with a wedge being the most common example.
Harrison and Williams \cite{harrison1987multidimensional} characterized when the stationary distribution admits an exponential form in smooth or polyhedral domain with drift and oblique reflection.
An integral expression for the Laplace transform in a wedge with oblique reflection was established by Franceschi and Raschel \cite{franceschi2019integral} who also characterized when this expression can be simplified in recent work with Bousquet--M\'elou, Price, and Hardouin \cite{bousquet2025stationary}.

The aforementioned research on the stationary distribution differs in spirit from the present work: we are \emph{not} concerned with deriving an exact formula.
Rather, we establish qualitative worst-case bounds on the stationary distribution and the mixing time in terms of geometric parameters.
Closer in spirit is, for example, the work of Dai and Miyazawa \cite{dai2011reflecting}.
They studied asymptotics for the tails of the stationary distribution in the orthant.

\subsection{Further directions}
\label{sec: FutureWork}

It looks difficult to derive an all--encompassing, closed-form expression for the stationary distribution and mixing time for every possible configuration of the barriers and domain.
However, closed-form expressions may be attainable for specific configurations.
For example, in the case of the unit disk with a single circular semipermeable barrier at radius $r\in (0,1)$, the equilibrium distribution may be susceptible to exact analysis because of the radial symmetry.

Bobrowski, Ka{\'z}mierczak, and Kunze \cite{bobrowski2017averaging} showed that the snapping-out Brownian motion exhibits state space collapse in the limit of low permeability. 
Specifically, they showed that the dynamics of a snapping-out Brownian motion converges to that of a continuous-time Markov chain with finite state space corresponding to the different components of the domain separated by the barriers.
The jump rates of the limiting chain were shown to only depend on the geometry of the barriers through their length and compartment sizes.
Next, it would be relevant to quantify how well the mixing time and stationary distribution are approximated by their analytically tractable limits when permeability is low but nonzero.  

Recall that our results include examples where the mixing time and stationary distribution decay exponentially as the domain grows large and other parameters are kept fixed.
Of course, that such exponential decay occurs in a worst-case example does not mean that it always materializes.
Clarifying general sufficient or necessary conditions to avoid such instances could be an interesting avenue of research.
Our proof methods can be used in some cases to get better bounds, but not always; see Remark \ref{rem: Adaptive}.

Additionally, our examples with exponential decay in Section \ref{sec: ResultSub} involve asymmetric permeabilities.
This means that it is easier to go from one side of the barrier to the other than in the reverse direction.
Another interesting open question is thus whether exponential decay can also occur if the permeabilities are symmetric. 

\section{Results}
\label{sec: Results}

\subsection{Snapping-out Brownian motion}
\label{sec: IntroRBMDef}

We adapt the setting of \cite{vanwerde2024recovering} and focus on the case where the process takes values in a planar domain with one-dimensional smooth curves as barriers.
If desired, however, one can also readily generalize our results and the following definitions with weaker regularity assumptions or to higher dimensions.

Consider an open planar set $D_0 \subseteq \bbR^2$ whose closure $D$ is connected, bounded, and has a smooth boundary $B_0 \de \partial D$.
Additionally, consider $C^\infty$-smooth curves $B_1,\ldots,B_{\nb}\subseteq D_0$.
In what follows, the $B_i$ with $i\geq 1$ will be semipermeable barriers and $B_0$ will be an impermeable barrier, keeping the process in $D$.

We adopt some assumptions on the geometry of the barriers.
First, assume that $B_i \cap B_j = \emptyset$ for every $i\neq j$ so that there are no intersection points.
Second, assume that the domain $D$ is simply connected so that $B_0$ is a single connected curve.
Finally, assume that the $B_i$ are simple closed curves so that there are no self-intersections or endpoints.
Then, in particular, the Jordan curve theorem yields that $\bbR^2\setminus B_i$ has precisely two connected components of which only one is bounded.

The process experiences orthogonal reflections upon hitting the barrier.
Let $\vec{n}_i:B_i \to \bbR^2$ be the unique vector field orthogonal to $B_i$, pointing towards the bounded component, and with unit length.
We say that $x\in \bbR^2$ is \emph{on the positive side of $B_i$} if $x$ lies in the closure of the bounded component of $\bbR^2 \setminus B_i$.
Similarly, the \emph{negative side} is the closure of the unbounded component.
Thus, $x\in B_i$ is both on the positive and negative sides.

Randomness from auxiliary processes is used to drive the movements of the process.
Denote $W_t$ for a planar Wiener process.
Further, fix scalars $\lambda_i^+,\lambda_i^->0$ for every $i\geq 1$, specifying the permeability of the two sides of each barrier and consider \cadlag\ continuous-time Markovian processes $s_i(t)$ taking values in $\{+1,-1 \}$ with transition rate $\lambda_i^+$ (resp.\ $\lambda_i^-$) from $-1$ to $+1$ (resp.\ from $+1$ to $-1$).
Finally, let $s_0(t) \de +1$ for all $t\geq 0$.

\pagebreak[3]
\begin{definition}\label{def: ReflectedBrownianMotion}
	Let $X_t$ and $\Lc{i}_t$ be continuous stochastic processes which take values in $D$ and $\bbR_{\geq 0}$, respectively.
	Then, $X_t$ is called \emph{a reflected Brownian motion with semipermeable barriers $B_i$ and local times $\Lc{i}_t$} if the following properties hold with probability one:
	\begin{enumerate}[leftmargin=2.5em, label = (\roman*)]
		\item\label{item: Def_ReflectedBrownianMotion_i}
		      $
			      \intd X_t =  \intd W_t +  \sum_{i=0}^{\nb} s_i(\Lc{i}_t) \vec{n}_i(X_t) \bb1\{X_t \in B_i \}\intd \Lc{i}_t.\nonumber
		      $
		\item\label{item: Def_ReflectedBrownianMotion_ii} For $i \leq {\nb}$, $\Lc{i}_t$ is nondecreasing, satisfies $\Lc{i}_0 =0$, and increases at time $t$ if and only if $X_t \in  B_i$.
		      That is,
		      \begin{align}
			      \Lc{i}_t = \int_0^t \bb1\{X_r\in B_i \}\, \intd \Lc{i}_r.\nonumber
		      \end{align}
		\item\label{item: Def_ReflectedBrownianMotion_iii} For $i \leq {\nb}$, if $s_i(\Lc{i}_t) = +1$ then $X_t$ is on the positive side of $B_i$.
		      Similarly, if $s_i(\Lc{i}_t) = -1$ then $X_t$ is on the negative side of $B_i$.
	\end{enumerate}
	The process $X_t$ may also be referred to as a \emph{snapping-out Brownian motion}.
\end{definition}
\pagebreak[2]
Processes satisfying Definition \ref{def: ReflectedBrownianMotion} are shown to exist pathwise uniquely given an initial condition in \cite{vanwerde2024recovering}.
Note that the law of the process $X_t$ not only depends on the initial condition $X_0$, but also on $s_i(0)$ if $X_0 \in B_i$.
For brevity, however, we will suppress this from the notation whenever it does not cause ambiguity.

\subsection{Environmental parameters}\label{sec: Parameters}
Consider lower and upper bounds on the permeabilities:
\begin{align}
	\lambda_{\min} & \de \min\bigl\{\lambda_i^*: i\in \{1,\ldots,{\nb} \}, \ * \in \{+,- \} \bigr\}, \label{eq:VagueUnit}        \\
	\lambda_{\max} & \de \max\bigl\{\lambda_i^*: i\in \{1,\ldots,{\nb} \},\,  *\in \{+,-\} \bigr\}\label{eq: Def_lambda_minmax}.
\end{align}
The relation between these parameters and the mixing properties is intuitive.
If $\lambda_{\min}$ is small, then one may stay stuck behind a barrier for a long time.
Further, if $\lambda_{\max}/\lambda_{\min}$ is large, then one may briefly transition to the other side but quickly get absorbed again.
Thus, large but asymmetric permeability can also be problematic.

The following two parameters are of a geometric nature, related to the difficulty of the configuration of barriers; see Figure \ref{fig: Pathologies}.
For every $i\leq {\nb} $ let $k_i:B_i \to \bbR_{\geq 0}$ be the \emph{unsigned curvature} of $B_i$.
Denote
\begin{align}
	\kappa \de \max\bigl\{ k_i(x): i\in \{0,1,\ldots,{\nb} \},\, x \in B_i \bigr\}. \label{eq: Def_kappa}
\end{align}
A bound on $\kappa$ rules out scenarios where the curves squiggle wildly.
It is further useful to ensure some minimal spacing between the barriers.
For $x\in \bbR^2$ let $\sB(x,r)$ denote the open ball of radius $r>0$ and set
\begin{align}
	\rho & \de \sup\bigl\{r \geq 0: \sB(x,r')\cap (\cup_{i=0}^{\nb} B_i)  \textnormal{ is connected }\forall r' \leq r,\, \forall x\in D\bigr\}. \label{eq: Def_rho}
\end{align}
This parameter serves to avoid scenarios where distinct barriers lie extremely close to each other, or scenarios where a single barrier doubles over with a U-bend and thus has a small separation from itself.
Such cases could be problematic for the mixing, as they may result in the effective permeability of a region of the domain being lowered.

\begin{figure}[h!]
	\centering
	\includegraphics[width = 0.9\textwidth]{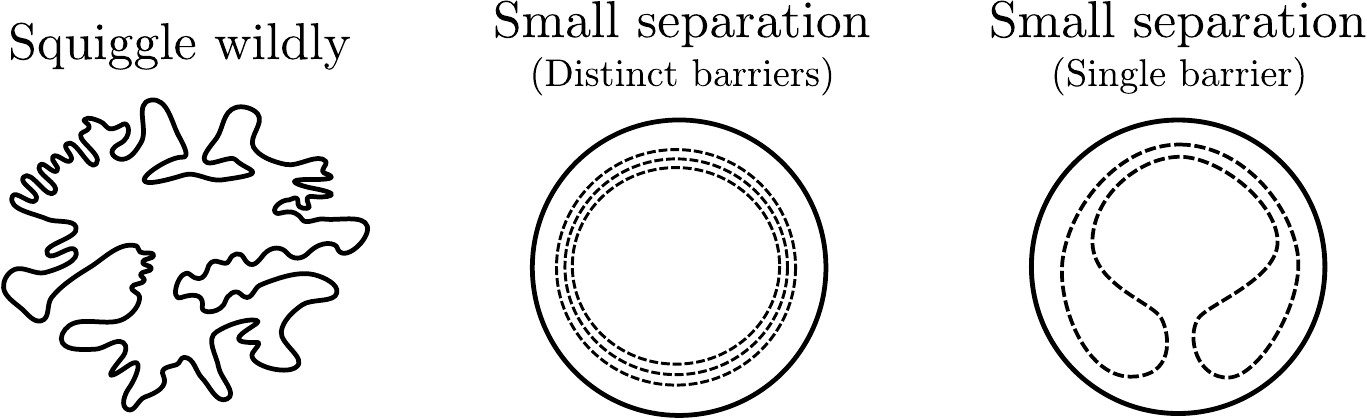}
	\caption{%
		Examples of pathological geometries that could complicate the mixing properties.
		A bound on $\kappa$ allows ruling out the leftmost example, but is not sufficient to rule out the two examples on the right.
		This is where a bound on $\rho$ is required.
	}
	\label{fig: Pathologies}
\end{figure}

Finally, it is natural that the size of the domain may be relevant for the mixing properties.
To quantify this, we consider the area $\Area(D)$ and the diameter in the geodesic distance:
\begin{align}
	\Delta \de \sup_{x,y\in D}\inf_\gamma \int_0^1 \Vert \gamma'(t) \Vert \, \intd t. \label{eq: Def_Delta}
\end{align}
Here, the infimum runs over piecewise smooth curves $\gamma:[0,1]\to D$ with $\gamma(0) = x$ and $\gamma(1) =y$ and $\Vert \cdot \Vert$ denotes the Euclidean norm on $\bbR^2$.

\subsection{Results}\label{sec: ResultSub}
Let $\pi$ denote the stationary distribution of the process $X_t$.
We are concerned with the \emph{mixing time}
\begin{align}
	\tmix \de \inf\bigl\{t\geq 0 :  \lvert \bbP(X_t \in E \mid X_0 = x_0) - \pi(E) \rvert \leq 1/4  \textnormal{ for all }x_0,\, E\bigr\}.\label{eq: Def_tmix}
\end{align}
Here, $x_0\in D$ runs over all possible initial conditions and $E\subseteq D$ over all measurable subsets with nonzero area.
Further, we consider the quantity
\begin{align}
	\pi_{\min} \de \inf\bigl\{\pi(E)/\Area(E) :\textnormal{measurable }E\subseteq D \textnormal{ with }\Area(E)>0 \bigr\}.\label{eq: Def_pimin}
\end{align}
to quantify if there are infrequent regions in the domain.
Theorem \ref{thm: Main_mixing} gives worst-case bounds on $\tmix$ and $\pi_{\min}$ in terms of the parameters from Section \ref{sec: Parameters}:

\begin{theorem}
	\label{thm: Main_mixing}
	There exists an absolute constant $c>0$ so that with $R \de c\min\{1/\kappa, \allowbreak 1/\lambda_{\max}, \allowbreak \rho \}$,
	\begin{align}
		\tmix \leq  \Delta^2 (R \lambda_{\min} )^{-\Delta/R} \quad \text{ and } \quad \pi_{\min} \geq (R \lambda_{\min})^{\Delta/R}/\Area(D).  \label{eq:KindNose}
	\end{align}
\end{theorem}

Observe that an exponential decay occurs in the quality of the bounds if $\Delta/R \to \infty$ in Theorem \ref{thm: Main_mixing}.
While it is possible that the exact formulation of the estimates could still be improved, as we make a few rough estimates to simplify notation throughout the proofs, this qualitative behavior of exponential decay is necessary as a worst-case estimate.

\pagebreak[3]
{Indeed, suppose that the barriers are a sequence of $2n+1$ nested circles, as displayed}
\begin{wrapfigure}[12]{r}{0.26\textwidth}
	\centering
	\vspace{-0.65em}
	\includegraphics[width=0.25\textwidth]{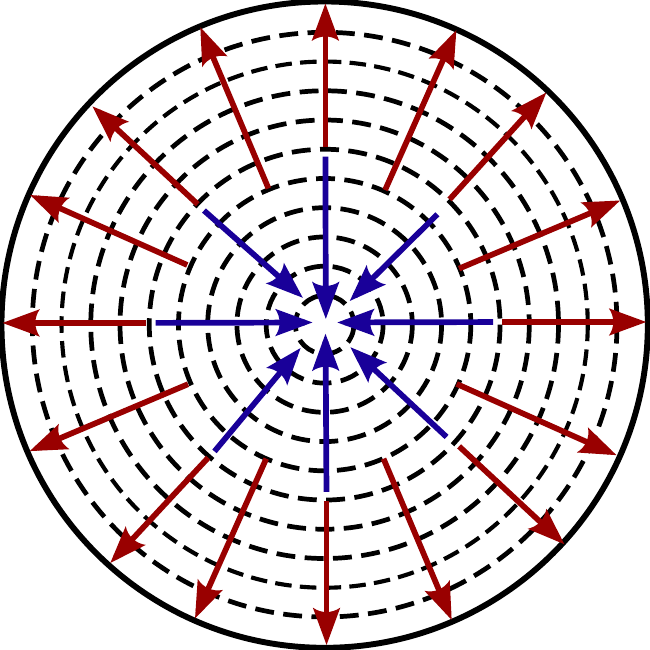}
	\vspace{-0.45em}
	\caption{}
	\label{fig: WorstCase}
	\vspace{0.2em}
\end{wrapfigure}
in Figure \ref{fig: WorstCase}:
\begin{align}
	B_i \de \{x \in \bbR^2: \Vert x \Vert = (2n+1 -i) \}\textnormal{ for }i=0,1,\ldots,2n\nonumber  \qquad \qquad \qquad  \qquad  \qquad  \qquad
\end{align}
where $n\geq 1$ is an integer.
Then, choosing the permeabilities appropriately so that $\lambda_i^+$ is somewhat smaller than $\lambda_{i}^-$ for $ i\leq n$ and greater for $i>n$, it can be ensured that there is effectively an inwards (resp.\ outwards) drift when $X_{t}$ is in the inner (resp.\ outer) half of the domain.
This gives rise to metastable behavior as $n\to \infty$ where $X_t$ stays stuck in the half of the domain where it starts for an exponential amount of time.
This demonstrates the necessity of exponential behavior in $\tmix$ as $\Delta/R\to \infty$ since the diameter $\Delta$ here grows linearly in $n$ while the other parameters remain fixed.

A similar example is applicable for the minimal value of the stationary distribution $\pi_{\min}$: simply take $\lambda_i^+$ smaller than $\lambda_i^-$ for all $i\leq 2n$.
Then, the resulting effective outwards drift will ensure that the stationary distribution becomes exponentially small near the center of the disc as $n\to \infty$.

\section{Proof of \texorpdfstring{Theorem \ref{thm: Main_mixing}}{Theorem}}\label{sec: ProofMainMixing}
The strategy is to reduce to the Doeblin minorization condition; see Lemma \ref{lem: ReductionBall}.
To verify the latter, it suffices to determine some $T >0$ and $c>0$ such that
\begin{equation}
	\bbP(X_T \in \sB(y,\varepsilon) \mid X_0 = x_0) \geq c\varepsilon^2
	\label{eqn:lower_bound_goal_towards_Doeblin_minorization}
\end{equation}
for every $y\in D$ and small $\varepsilon>0$.

Section \ref{sec: CloseEndpoints} estimates \eqref{eqn:lower_bound_goal_towards_Doeblin_minorization} in the case where $x_0$ and $y$ are sufficiently distant from all barriers, but not necessarily on the same side, for the special case where $x_0$ and $y$ lie close together.
Section \ref{sec: DistantEndpoints} next applies this repeatedly to also cover points that are distant from each other yet still at nontrivial distance from the barriers.

It then remains to consider the scenario where $x_0$ or $y$ could lie close to a barrier.
This is done in Section \ref{sec: ReductionClose} by reducing to the case where they are distant from the barriers.
The reduction for $x_0$ is fairly straightforward, but the reduction for the endpoints $y$ requires an additional trick exploiting the time-reversibility of classical reflected Brownian motion without semipermeable barriers; see Lemma \ref{lem: TrickTimeRev}.
We conclude in Section \ref{sec: ProofThm}.

The reader may find it helpful to start by glossing over the figures in this section to get a global idea of the proof approach.
In particular, the informal ideas involved with the local estimates when $x_0$ and $y$ lie close to one another are visualized in Figures \ref{fig: ForcedSmash} and \ref{fig: AvoidBarrier}.
That a sequence of local steps can be used for a long-range estimate is visualized in Figure \ref{fig: Geo}.


\subsection{Points distant from the barriers and close together}\label{sec: CloseEndpoints}

Fix some small $\delta >0$ and introduce a characteristic scale of distance by
\begin{align}
	r_\delta \de \delta\min\{1/\kappa, 1/\lambda_{\max},\rho \}. \label{eq: Def_R_eps}
\end{align}
Denote $\cD(\delta)$ for the set of points in $D$ of distance $\geq  r_\delta/5$ from all barriers:
\begin{align}
	\cD(\delta) \de  \bigl\{x\in D: \Vert x - z \Vert \geq   r_\delta/5 ,\ \forall z \in \cup_{i=0}^m B_i  \bigr\}.\label{eq: Def_Deps}
\end{align}
Our goal here is to estimate $\bbP(X_t \in \sB(y,\varepsilon) \mid X_0 = x_0)$ for all $x_0,y \in \cD(\delta)$ which are sufficiently close together, say $\Vert x_0-y \Vert \leq (4/5)r_\delta$.

There are two main challenges.
First, if $x_0$ and $y$ lie on different sides of a barrier, then it has to be shown that the barrier is crossed with nontrivial probability.
Second, it has to be shown that we move from $x_0$ to a small neighborhood of $y$ with nontrivial probability when they lie on the same side of all barriers.
The first challenge is dealt with in Lemma \ref{lem: ReductionSameSide}, the second in Lemma \ref{lem: TransportSameSide}, and we combine the estimates in Corollary \ref{cor: TransportClose}.
The proofs for both lemmas use the following elementary property of the driving Wiener process; see Figure \ref{fig: PillShaped} for a visualization.

\begin{lemma}\label{lem: PillShaped}
	For every $\gamma \in (0,1)$ there exists a constant $c>0$ depending only on $\gamma$ such that the following holds.
	Consider $R>0$ and $y\in \bbR^2$ with $\Vert y \Vert \leq R$.
	Then, for every $\varepsilon \leq \gamma R/2$,
	\begin{align}
		\bbP\bigl(W_{R^2} \in \sB(y, \varepsilon) \text{ and } \Vert W_t - (t/R^2)y \Vert \leq \gamma R, \, \forall t\leq R^2 \bigr) \geq c(\varepsilon/R)^2. \label{eq:TediousFish}
	\end{align}
\end{lemma}

\begin{proof}
	Let $c_1 \de \exp(-2)/2\pi$ be the minimum of the density of a $2$-dimensional Gaussian random vector on the ball of radius $2$.
	Then, for every $\varepsilon\leq R$, by the scaling principle,
	\begin{align}
		\bbP(W_{R^2} \in \sB(y,\varepsilon)) = \bbP(W_{1} \in \sB(y/R,\varepsilon/R)) \geq c_1(\varepsilon/R)^2.
	\end{align}
	Hence, conditioning on $W_{R^2}$, it suffices to show that there exists $c_2 >0$ such that
	\begin{align}
		\inf_{x \in \sB(y,\varepsilon)}\bbP\bigl( \Vert W_t - (t/R^2)y \Vert \leq \gamma R ,\, \forall t\leq R^2 \mid W_0 = 0, W_{R^2} = x \bigr) \geq c_2.\label{eq:ValidInk}
	\end{align}
	Fix some $x \in \sB(y,\varepsilon)$ and let $\fB_t \de W_t - (t/R^2)x$ for every $t\leq R^2$.
	Then, conditional on $W_{R^2} = x$, the process $\{\fB_t:0\leq  t\leq R^2 \}$ is a two-dimensional Brownian bridge pinned to $0$ at times $0$ and $R^2$.
	Further, using that $\varepsilon\leq  \gamma R/2$ with the triangle inequality, we have $\Vert W_t - (t/R^2)y \Vert \leq \gamma R$ whenever $\Vert \fB_t \Vert \leq \gamma R/2$.
	Hence,
	\begin{align}
		\bbP\bigl(\Vert W_t - (t/R^2)y \Vert \leq \gamma R  , \forall t\leq{} & {} R^2 \mid W_0 = 0, W_{R^2} = x \bigr)\label{eq:MagicGum}                                                        \\
		                                                                      & \geq \bbP\bigl(\Vert \fB_t \Vert \leq \gamma R/2, \forall t\leq R^2\mid \fB_0 = 0, \fB_{R^2} = 0 \bigr).\nonumber
	\end{align}
	The scaling principle allows one to eliminate the $R$-dependence from the right-hand side of \eqref{eq:MagicGum}, giving a nonzero constant depending only on $\gamma$.
	This proves \eqref{eq:ValidInk}, as desired.
\end{proof}
\begin{figure}[hbt]
	\centering
	\includegraphics[width = 0.8\textwidth]{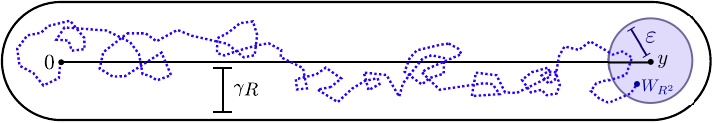}
	\caption{
		Visualization of the event described in the probability of \eqref{eq:TediousFish} from Lemma \ref{lem: PillShaped}.
		This event asks that the Wiener process $W_t$, represented by the dotted blue line, progresses in a straight line towards $y$ up to a moderate deviation of distance $\leq \gamma R$ and that $W_{R^2}$ ends up in a neighborhood of radius $\varepsilon$ of $y$.
		Here, $\varepsilon >0$ could be arbitrarily small.
	}
	\label{fig: PillShaped}
\end{figure}

\subsubsection{Crossing a barrier}
Consider the case where $x_0$ and $y$ lie on different sides of a barrier.
Let $\cZ(\delta,y)$ be the set of points $z\in \cD(\delta/2)$ at distance $\leq (4/5) r_\delta$ from $y$ and on the same side of every barrier as $y$.
Then, Lemma \ref{lem: ReductionSameSide} shows that $X_{ r_\delta^2} \in \cZ(\delta,y)$ with nonzero probability.

The proof idea is described in the caption of Figure \ref{fig: ForcedSmash}.
We subdivide the execution of this idea into a sequence of smaller lemmas.
The first step is to show that $s_i$ and $W_t$ have certain convenient properties with nonzero probability:
\begin{lemma}\label{lem: W_si_nice}
	For every $\alpha, \beta,\gamma >0$, there exists a constant $c>0$ such that the following holds for every $\delta < 1/2$.
	Consider $x_0,y \in \cD(\delta)$ with $\Vert x-y \Vert \leq (4/5) r_\delta$ such that $x_0$ and $y$ lie on different sides of some barrier $B_i$ with $i\geq 1$.
	Define
	\begin{align}
		\cE_1 \de \inf\{t\geq 0: s_i(t) \neq s_i(0) \} \quad \text{ and }\quad \cE_2 \de \inf\{t\geq \cE_1: s_i(t)\neq  s_i(\cE_1)\} - \cE_1 \nonumber
	\end{align}
	to be the first two increments between times for which $s_i$ changes value.
	Then,
	\begin{align}
		\bbP\bigl( \cE_1 <  \alpha r_\delta, \,  \cE_2 >  \beta r_\delta, \,  \Vert   W_t - (t/ r_\delta^2) (y- x_0) \Vert \leq   \gamma r_\delta,\, \forall t\leq  r_\delta^2 \bigr) \geq c  r_\delta \lambda_{\min}.  \label{eq:SlowIce}
	\end{align}
\end{lemma}

\begin{proof}
	Note that $\cE_1$ and $\cE_2$ are independent exponential random variables with rate $\lambda_i^+$ or $\lambda_i^-$, depending on what side of the barrier $x_0$ lies.
	In particular, the rate of $\cE_1$ is at least $\lambda_{\min}$ and the rate of $\cE_2$ is at most $\lambda_{\max}$.
	Hence,
	\begin{align}
		\bbP(\cE_1 <  \alpha r_\delta \text{ and }\cE_2 >  \beta r_\delta) & = \bbP(\cE_1 <  \alpha r_\delta) \bbP(\cE_2 > \beta r_\delta)\label{eq:MythicalCell}                      \\
		                                                                   & \geq \bigl(1-\exp(-\lambda_{\min}  \alpha r_\delta) \bigr)\exp(- \beta \lambda_{\max} r_\delta).\nonumber
	\end{align}
	Note that $1-\exp(-x) \leq x$ for $x>0$ and note that $\lambda_{\max} r_\delta  < 1$ by \eqref{eq: Def_R_eps}.
	It hence follows from \eqref{eq:MythicalCell} that with $c_1 \de \alpha \exp(-\beta)$,
	\begin{align}
		\bbP(\cE_1 <  \alpha r_\delta \text{ and }\cE_2 > \beta r_\delta) \geq c_1  r_\delta \lambda_{\min}.\label{eq:HauntedRug}
	\end{align}
	We next apply Lemma \ref{lem: PillShaped} with $R =  r_\delta$.
	The additional parameter $\delta$ in Lemma \ref{lem: PillShaped} is not important here, let us take $\delta = \gamma R/2$.
	Then, we find $c_2 >0$ depending only on $\gamma$ such that
	\begin{align}
		\bbP\bigl(\Vert  W_t - (t/ r_\delta^2) (x_0 - y) \Vert \leq   \gamma r_\delta,\, \forall t\leq  r_\delta^2\bigr) \geq c_2.\label{eq:PinkNose}
	\end{align}
	Combine \eqref{eq:HauntedRug} with \eqref{eq:PinkNose} and use that $W_t$ is independent of $s_i$ by assumption to conclude.
\end{proof}

Recall from \eqref{eq: Def_kappa} that $\kappa$ is an upper bound on the curvature of the barriers.
In particular, the barriers are well-approximated by straight lines at the scale of $r_{\delta}$.
The following lemma makes this precise and will be used repeatedly:
\begin{lemma}[{Lemma 3.5 in \cite{vanwerde2024recovering}}]\label{lem: BarrierLocallyStraight}
	Suppose that $\delta < 1/2$.
	Then, for every $x_0 \in \bbR^2$, there is at most one barrier $B_i$ which intersects $\sB(x_0,r_{\delta})$.
	Moreover, if such a $B_i$ exists, then there exists a unit vector $ \hatn \in \bbR^2$ depending on $x_0$ and $B_i$ with the following properties:
	\begin{enumerate}[leftmargin = 2em, label = (\arabic*)]
		\item\label{item:GhostlyQuip} There exists some $\mathfrak{c}\in \bbR$ such that $\lvert \langle  y, \hatn\rangle - \mathfrak{c} \rvert <  4\delta r_\delta$ for every $y\in B_i\cap \sB(x_0, r_\delta)$.
		\item\label{item:IdleBat} One has $\Vert \vec{n}_i(y)-  \hatn \Vert < 2\delta$ for every $y\in B_i \cap \sB(x_0, r_\delta)$.
		\item\label{item:JollyDragon} Every point $z \in \sB(x_0, r_\delta)$ on the positive side of $B_i$ satisfies $ \langle  z, \hatn\rangle > \mathfrak{c} -  4\delta r_\delta$.
		      Similarly, it holds for every $z\in \sB(x_0, r_\delta)$ on the negative side of $B_i$ that $ \langle  z, \hatn\rangle < \mathfrak{c}  +  4\delta r_\delta$.
	\end{enumerate}
\end{lemma}

From here on, we adopt the assumptions and notation of Lemma \ref{lem: W_si_nice} and assume that the event described in the probability on the left-hand side of \eqref{eq:SlowIce} occurs.
In particular, we fix $x_0, y\in\cD(\delta)$ with $\Vert x_0 - y \Vert\leq (4/5) r_\delta$ and let $i\geq 1$ be the index of the barrier $B_i$ which separates $x_0$ and $y$.
Let it be understood that we take $X_0 = x_0$.
Our goal in the subsequent Lemmas \ref{lem: StayInBall} to \ref{lem: NotAgain} is to deduce properties that follow from the event in \eqref{eq:SlowIce}.

\begin{lemma}[Stay in ball if local time is controlled]\label{lem: StayInBall}
	Adopt the setting of Lemma \ref{lem: W_si_nice} and assume that the event from \eqref{eq:SlowIce} occurs.
	Then, for every $t\leq  r_\delta^2$ it holds that $X_t\in \sB(x_0, r_\delta)$ if $\Lc{i}_{t} < (1/5 -\gamma)r_\delta$.
\end{lemma}
\begin{proof}
	If the claim was false, then we could find some minimal $t_* \leq r_\delta^2$ with $\Lc{i}_{t_*} < (1/5 -\gamma)r_\delta$ but $X_{t_*} \not\in \sB(x_0,r_\delta)$.
	Then, for every $t< t_*$ we have $X_t\in \sB(x_0, r_\delta)$ and
	Lemma \ref{lem: BarrierLocallyStraight} yields that $B_i$ is the only barrier which intersects $\sB(x_0, r_\delta)$.
	Hence, the stochastic differential equation from Definition \ref{def: ReflectedBrownianMotion} implies that
	\begin{align}
		\Vert X_t - x_0 \Vert \leq \Vert W_t \Vert +  \Lc{i}_t \text{ for every }t < t_*. \label{eq:GladToy}
	\end{align}
	Here, using \eqref{eq:SlowIce} and the assumption that $\Vert x_0 - y \Vert \leq (4/5) r_\delta$,
	\begin{align}
		\Vert W_t \Vert \leq  \Vert  W_t - (t/ r_\delta^2) (x_0 - y) \Vert + \Vert (t/ r_\delta^2) (x_0 - y) \Vert \leq  \gamma r_\delta + (4/5) r_\delta.  \label{eq:OddKing}
	\end{align}
	The combination of \eqref{eq:GladToy} and \eqref{eq:OddKing} with continuity of the processes involved implies that $\Vert X_{t_*} - x_0 \Vert \leq \gamma r_\delta + (4/5)r_{\delta}+  \Lc{i}_{t_*}$.
	Using that $\Lc{i}_{t_*} < (1/5 - \gamma)r_\delta$ now yields that $X_{t_*} \in \sB(x_0,r_\delta)$, a contradiction.
	This concludes the proof.
\end{proof}

\begin{lemma}[Barrier is crossed at least once]\label{lem: AtLeastOnce}
	Adopt the setting of Lemma \ref{lem: W_si_nice} and assume that the event from \eqref{eq:SlowIce} occurs.
	Additionally, assume that $\alpha < 1/5 - \gamma$.
	Then, there exists some $t \leq  r_\delta^2$ such that $s_i(\Lc{i}_t) \neq s_i(0)$.
\end{lemma}
\begin{proof}
	Suppose not, meaning that $s_i(\Lc{i}_t) = s_i(0)$ for every $t \leq  r_\delta^2$.
	Then, since $\cE_1 \leq  \alpha_1 r_\delta$, it follows that $\Lc{i}_{ r_\delta^2} \leq  \alpha r_\delta < (1/5 - \gamma)r_\delta$.
	Lemma \ref{lem: StayInBall} now implies that $X_t \in \sB(x_0,  r_\delta)$ for every $t\leq  r_\delta^2$ and consequently $B_i$ is the only barrier which intersects $\sB(x_0, r_\delta)$ by Lemma \ref{lem: BarrierLocallyStraight}.
	The stochastic differential equation in Definition \ref{def: ReflectedBrownianMotion} hence yields
	\begin{align}
		\Vert X_{  r_\delta^2} - x_0 - W_{  r_\delta^2} \Vert \leq\Lc{i}_{  r_\delta^2}  \leq  \alpha r_\delta < (1/5 - \gamma)r_\delta.\label{eq:GladSax}
	\end{align}
	On the other hand, using the triangle inequality as well as \eqref{eq:SlowIce},
	\begin{align}
		\Vert X_{  r_\delta^2} - x_0 - W_{  r_\delta^2} \Vert \geq \Vert X_{  r_\delta^2}- y \Vert - \Vert W_{ r_\delta^2} - (y -x_0) \Vert \geq \Vert X_{  r_\delta^2}- y \Vert -  \gamma r_\delta.\label{eq:DarkGoose2}
	\end{align}
	The assumption that $s_i(\Lc{i}_t) = s_i(0)$ implies that $X_{ r_\delta}^2$ lies on a different side of $B_i$ than $y$.
	(Recall that Lemma \ref{lem: W_si_nice} assumed that $x_0$ and $y$ lie on different sides.)
	Consequently, since $y \in \cD(\delta)$ implies that $y$ is at distance $\geq  r_\delta/5$ from $B_i$, we have $
		\Vert X_{ r_\delta^2 } - y\Vert \geq  r_\delta / 5 $.
	Combine this with \eqref{eq:DarkGoose2} to find that $\Vert X_{  r_\delta^2} - x_0 - W_{  r_\delta^2} \Vert \geq (1/5 - \gamma) r_{\delta}$, contradicting \eqref{eq:GladSax}.
\end{proof}

\begin{lemma}[Barrier is not crossed a second time]\label{lem: NotAgain}
	Adopt the notion and assumptions of Lemma \ref{lem: W_si_nice} and assume that the event from \eqref{eq:SlowIce} occurs.
	Additionally, assume that
	\begin{align}
		\alpha + \beta < 1/5 - \gamma\ \text{ and }\ (1-2\delta)\beta -2\gamma - 12 \delta \geq 4\delta.\label{eq:ZombieGum}
	\end{align}
	Then, it holds with $T\de \inf\{t \geq 0: s_i(\Lc{i}_t) \neq s_i(0) \}$ that
	$
		\Lc{i}_{ r_\delta^2} \leq \Lc{i}_{T} +  \beta r_\delta.
	$
	In particular,
	\begin{align}
		s_i(\Lc{i}_t) = s_i(\Lc{i}_{T}) \text{ for every } t\in [T,  r_\delta^2].
	\end{align}
\end{lemma}
\begin{proof}
	The bound on $\cE_2$ in \eqref{eq:SlowIce} implies that $s_i(\Lc{i}_t) = s_i(\Lc{i}_{T})$ for every $t\in [T,  r_\delta^2]$ with $\Lc{i}_t \leq \Lc{i}_T + \beta r_\delta$.
	It hence suffices to prove this bound on the local time.

	Suppose to the contrary that $t_* \de \inf\{t\geq T: \Lc{i}_t > \Lc{i}_{T} +  \beta r_\delta\}$ is less than or equal to $ r_\delta^2$.
	We will derive a contradiction.
	Since $\Lc{i}_{t_*} - \Lc{i}_{T} =  \beta r_\delta$ by definition of $t_*$ and since $\Lc{i}_{T} \leq  \alpha r_\delta$ by the bound on $\cE_1$ in \eqref{eq:SlowIce},
	\begin{align}
		\Lc{i}_{t_*}= \Lc{i}_{T} + (\Lc{i}_{t_*} - \Lc{i}_{T}) = (\alpha + \beta)r_\delta < (1/5 - \gamma)r_\delta.\label{eq:WeepyGoose}
	\end{align}
	Lemma \ref{lem: StayInBall} then implies that $X_t\in \sB(x_0,  r_\delta)$ for every $t\leq t_*$ so that Lemma \ref{lem: BarrierLocallyStraight} implies that $B_i$ is the only barrier with which the process may interact.
	Hence, also using that $s_i(\Lc{i}_t) = s_i(\Lc{i}_T)$ for every $t\in [T,t_*)$ by definition of $t_*$ and the bound on $\cE_2$ in \eqref{eq:SlowIce}, the stochastic differential equation in Definition \ref{def: ReflectedBrownianMotion} yields that
	\begin{align}
		X_{t_*}  =X_T + (W_{t_*} - W_T) + s_i(\Lc{i}_T)\int_{T}^{t_*}\bb1\{X_t \in B_i \} \vec{n}_i(X_t)\, \intd \Lc{i}_t.  \label{eq:GreenOtter}
	\end{align}
	We next study the right-hand side of \eqref{eq:GreenOtter} with as main goal to prove that $X_{t_*}\not\in B_i$.

	Note that $T$ being the time when $s_i(\Lc{i}_T)$ changes sign necessitates that $X_{T} \in B_i$.
	Hence, by item \ref{item:GhostlyQuip} from Lemma \ref{lem: BarrierLocallyStraight}, with $\mathfrak{c}$ and $\hatn$ as in that lemma,
	\begin{align}
		\langle X_T, s_i(\Lc{i}_T)\hat{n} \rangle \geq  s_i(\Lc{i}_T) \mathfrak{c} - 4\delta r_\delta.
	\end{align}
	Further, recall that $X_t\in \sB(x_0,r_\delta)$ for $t\leq t_*$ so that item \ref{item:IdleBat} in Lemma \ref{lem: BarrierLocallyStraight} implies that $\langle \vec{n}_i(X_t) , \hatn \rangle \geq 1- 2\delta$ for every $t\leq t_*$ with $X_t \in B_i$.
	Consequently,
	\begin{align}
		\Bigl\langle \int_{T}^{t_*}\bb1\{X_t \in B_i \} \vec{n}_i(X_t)\, \intd \Lc{i}_t, \hatn \Bigr\rangle \geq (1 - 2\delta)\bigl(\Lc{i}_{t_*} - \Lc{i}_T\bigr) \geq (1-2\delta)\beta r_\delta.
	\end{align}
	Recall from the setup of Lemma \ref{lem: W_si_nice} that $y$ lies on a different side of $B_i$ than $x_0$ and that $\Vert x-y \Vert \leq (4/5)r_\delta$ so that $y\in \sB(x_0, r_\delta)$.
	item \ref{item:JollyDragon} in Lemma \ref{lem: BarrierLocallyStraight} hence yields that
	$
		\langle x_0, s_i(0) \hatn \rangle > s_i(0)\mathfrak{c}  - 4\delta r_\delta
	$
	and
	$
		\langle y, s_i(0) \hatn \rangle < s_i(0)\mathfrak{c} + 4\delta r_\delta.
	$
	Consequently, recalling from \eqref{eq:SlowIce} that $\Vert W_t - (t/r_\delta^2)(y - x_0) \Vert \leq \gamma r_\delta$ for every $t\leq r_\delta^2$ and using that $s_i(\Lc{i}_T) = -s_i(0)$,
	\begin{align}
		\langle W_{t_*} - W_T,  s_i(\Lc{i}_T) \hatn \rangle & \geq -2\gamma r_\delta -  \Bigl(\frac{t_*}{r_{\delta}^2} - \frac{T}{r_{\delta}^2} \Bigr)\langle y- x_0,  s_i(0) \hatn \rangle \geq  -(2\gamma + 8 \delta)r_\delta.\label{eq:VagueBee}
	\end{align}
	The second inequality here also used that $T < t_*  \leq r_\delta^2$.

	Combine \eqref{eq:GreenOtter}--\eqref{eq:VagueBee} and use the assumption in \eqref{eq:ZombieGum} to find that
	\begin{align}
		\langle X_{t_*}, s_i(\Lc{i}_T) \rangle
		\geq s_i(\Lc{i}_T) \mathfrak{c} + \bigl((1-2\delta)\beta -2\gamma - 12 \delta  \bigr)r_\delta
		\geq s_i(\Lc{i}_T) \mathfrak{c} + 4 \delta r_\delta.
	\end{align}
	Then, item \ref{item:GhostlyQuip} in Lemma \ref{lem: BarrierLocallyStraight} implies that $X_{t_*} \not\in B_i$.
	However, the definition that $t_* = \inf\{t\geq T: \Lc{i}_t > \Lc{i}_T + \beta r_\delta \}$ implies that $\Lc{i}_{t_*} = \Lc{i}_T + \beta r_\delta$ and that for every small $\varepsilon >0$ it holds that $\Lc{i}_{t_* + \varepsilon} > \Lc{i}_T + \beta r_\delta$.
	This yields a contradiction since Definition \ref{def: ReflectedBrownianMotion} states that the local time can only increase when process is on the barrier.
	This concludes the proof.
\end{proof}

Recall that $\cZ(\delta,y)$ denotes the set of points $z\in \cD(\delta/2)$ at distance $\leq (4/5) r_\delta$ from $y$ and on the same side of every barrier as $y$.
We are now prepared to show that $X_t$ lands in $\cZ(\delta, y)$ with nontrivial probability:
\begin{lemma}
	\label{lem: ReductionSameSide}
	There exist absolute constants $\delta_0,c>0$ such that the following holds for every $\delta \leq \delta_0$.
	Consider $x_0,y \in \cD(\delta)$ with $\Vert x-y \Vert \leq (4/5) r_\delta$ such that $x_0$ and $y$ lie on different sides of some barrier $B_i$ with $i\in\{1,\ldots,k \}$.
	Then,
	\begin{align}
		\bbP(X_{ r_\delta^2} \in \cZ(\delta,y) \mid X_0 = x_0) \geq c  r_\delta \lambda_{\min}.
	\end{align}
\end{lemma}
\begin{proof}
	Let $\alpha, \beta,\gamma, \delta_0 >0$ be such that the constraints in \eqref{eq:ZombieGum} are satisfied for every $\delta \leq \delta_0$ and $\alpha + \beta < 1/10 - \gamma$.
	Such values exist since one can first fix $\beta$ at some value $<1/15$ and subsequently take all other parameters sufficiently small.
	Then, by Lemma \ref{lem: W_si_nice}, it suffices to show that $X_t \in \cZ(\delta,y)$ whenever the event from \eqref{eq:SlowIce} occurs.

	Combining Lemmas \ref{lem: AtLeastOnce} and \ref{lem: NotAgain} yields that $X_t$ crossed $B_i$ exactly once.
	Hence, $X_{ r_\delta^2}$ is on the same side of $B_i$ as $y$.
	Further, since $\Lc{i}_T \leq  \alpha r_\delta$ by the bound on $\cE_2$ in \eqref{eq:SlowIce} and $\Lc{i}_{r_\delta^2} - \Lc{i}_T \leq \beta r_\delta$ by Lemma \ref{lem: NotAgain},
	\begin{align}
		\Lc{i}_{ r_\delta^2} \leq  (\alpha  + \beta)r_\delta  < (1/10 - \gamma)r_\delta < (1/5 - \gamma)r_\delta .\label{eq:ZombieBeetle}
	\end{align}
	Lemma \ref{lem: StayInBall} then yields that $X_{t}$ is in $\sB(x_0,  r_\delta)$ for every $t\leq  r_\delta^2$.
	Hence, since $B_i$ is the only barrier which intersects $\sB(x_0,  r_\delta)$ by Lemma \ref{lem: BarrierLocallyStraight}, it follows that $X_{ r_\delta^2}$ is on the same side of every barrier as $y$, also for barriers $B_j$ with $j\neq i$.
	Further, by the triangle inequality,
	\begin{align}
		\Vert X_{ r_\delta^2} -y \Vert \leq \Vert X_{ r_\delta^2} - (x_0 + W_{ r_\delta^2}) \Vert + \Vert (x_0 + W_{ r_\delta^2}) - y \Vert.
	\end{align}
	Here, using the stochastic differential equation from Definition \ref{def: ReflectedBrownianMotion} together with the second-to-last inequality in \eqref{eq:ZombieBeetle}, we have
	$
		\Vert X_{ r_\delta^2} - (x_0 + W_{ r_\delta^2}) \Vert  \leq \Lc{i}_{ r_\delta^2} \leq (1/10 - \gamma)r_\delta.
	$
	Further, \eqref{eq:SlowIce} yields $\Vert (x_0 + W_{ r_\delta^2}) - y \Vert \leq  \gamma r_\delta$.
	Hence, we have that
	\begin{align}
		\Vert X_{ r_\delta^2} - y \Vert\leq (1/10) r_\delta \leq (4/5)r_\delta.\label{eq:UnripePaint}
	\end{align}
	Finally, recall from \eqref{eq: Def_Deps} that $y\in \cD(\delta)$ means that $y$ is at distance $\geq  r_\delta/5$ from all barriers.
	It hence follows from the first inequality in \eqref{eq:UnripePaint} that $X_{ r_\delta^2}$ is at distance $\geq  r_\delta/10$ from all barriers.
	This means that $X_{ r_\delta^2} \in \cD(\delta/2)$.

	We have shown that the event in \eqref{eq:SlowIce} implies that $X_{r_\delta^2}$ is a point in $\cD(\delta/2)$ at distance $\leq (4/5)r_\delta$ from $y$ and on the same side of every barrier.
	Thus, $X_{r_\delta^2} \in \cZ(\delta,y)$ as desired.
\end{proof}

\begin{figure}[hbt]
	\centering
	\includegraphics[width = 0.85\textwidth]{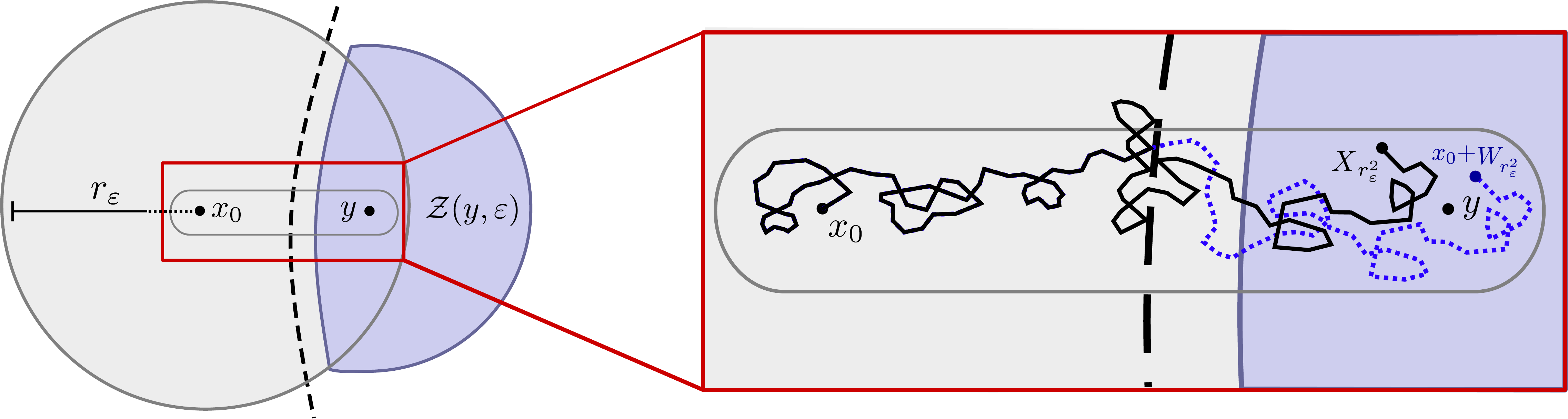}
	\caption{
		Visualization of Lemma \ref{lem: ReductionSameSide} and its proof idea.
		Using Lemma \ref{lem: PillShaped}, we can ensure that $X_t$ runs into the barrier with nontrivial probability.
		Then, also using that $\lambda_{\min}$ and $\lambda_{\max}$ control how quickly $s_i$ changes, we can estimate the probability that $X_t$ crosses the barrier exactly once.
		This allows proving that $X_{ r_\delta^2}$ lands in $\cZ(\delta,y)$ with nontrivial probability.
	}
	\label{fig: ForcedSmash}
\end{figure}

\subsubsection{Landing in a small neighborhood of \texorpdfstring{$y$}{y} starting from same side of barrier}
We here study the probability that $X_{t} \in \sB(y,\varepsilon)$ when $x_0,y \in \cD(\delta/2)$ are on the same side of every barrier.
The proof idea is outlined in the caption of Figure \ref{fig: AvoidBarrier}.
We start with a preliminary geometric lemma:
\begin{lemma}[Tube does not intersect barrier]\label{lem: StripIntersect}
	There exist absolute constants $\gamma_0,\delta_0 >0$ such that the following holds for every $\gamma \leq \gamma_0$ and $\delta \leq \delta_0$.
	Consider $x_0, y\in \cD(\delta/2)$ with $\Vert x_0 - y \Vert \leq (4/5) r_\delta$ and assume that $x_0$ and $y$ are on the same side of every barrier.
	Then, for every $i \geq 0$,
	\begin{align}
		B_i\cap \{z \in \bbR^2: \exists \alpha \in [0,1],\ \Vert  \alpha x_0 + (1-\alpha)y - z \Vert \leq   \gamma r_\delta \} = \emptyset. \label{eq:CrimsonRag}
	\end{align}
\end{lemma}
\begin{proof}
	Suppose that, on the contrary, there exists some $z\in B_i$ with $\Vert \alpha x_0 + (1- \alpha)y - z \Vert \leq \gamma r_\delta$ for some $\alpha \in [0,1]$.
	Then, in particular, $\Vert z - x_0 \Vert \leq (4/5)r_\delta + \gamma r_\delta$ so we may assume that $B_i$ intersects $\sB(x_0, r_\delta)$ by taking $\gamma_0 < 1/5$.

	Take $\delta_0 < 1/2$ and let $\mathfrak{c}$ and $\hat{n}$ be as in Lemma \ref{lem: BarrierLocallyStraight}.
	Recall that $x_0$ and $y$ are assumed to lie on the same side of $B_i$.
	Let us assume that they are on the positive side.
	The case where they are on the negative side proceeds similarly.
	Using item \ref{item:GhostlyQuip} in Lemma \ref{lem: BarrierLocallyStraight},
	\begin{align}
		\langle \alpha x_0 + (1-\alpha)y, \hat{n} \rangle \leq \langle z , \hat{n} \rangle + \gamma r_\delta \leq \mathfrak{c} + \gamma r_\delta + 4\delta r_\delta.\label{eq:JollyCobra}
	\end{align}
	This is only possible if the same inequality is satisfied by at least one of $x_0$ or $y$:
	\begin{align}
		\langle x_0, \hat{n} \rangle \leq  \mathfrak{c} + \gamma r_\delta+ 4\delta r_\delta \ \text{ and/or }\ \langle y, \hat{n} \rangle \leq \mathfrak{c}  + \gamma r_\delta + 4\delta r_\delta.
	\end{align}
	Let us assume that the inequality with $y$ holds.
	The case with $x_0$ is similar.

	Taking $\gamma_0$ and $\delta_0$ sufficiently small such that $\gamma - 8\delta < 1/10$, we may assume that the point $\tilde{y} \de y - (\gamma r_\delta - 8\delta r_\delta) \hat{n}$ satisfies $\Vert y - \tilde{y} \Vert < r_\delta/10$.
	In particular, it then holds that $\tilde{y}\in \sB(x_0, r_\delta)$ and further $\langle \tilde{y}, \hat{n}  \rangle \leq \mathfrak{c} - 4 \delta r_\delta$.
	Item \ref{item:JollyDragon} from Lemma \ref{lem: BarrierLocallyStraight} now yields that $\tilde{y}$ is not on the positive side of $B_i$.

	Considering that $y$ was on the positive side of $B_i$ and $\tilde{y}$ is not, there is at least one point on the line segment connecting $y$ to $\tilde{y}$ that lies in $B_i$.
	In particular, there is at least one point in $B_i$ at distance $< r_\delta /10$ from $y$.
	However, the assumption $y\in \cD(\delta/2)$ means that $y$ is at distance $\geq r_\delta /10$ from all barriers; recall \eqref{eq: Def_Deps}.
	This yields a contradiction.
\end{proof}

\begin{lemma}
	\label{lem: TransportSameSide}
	There exist absolute constants $\delta_0, \gamma, c>0$ such that the following holds for $\delta \leq \delta_0$.
	Consider $x_0,y \in \cD(\delta/2)$ with $\Vert x-y \Vert\leq (4/5) r_\delta$ such that $x_0$ and $y$ lie on the same side of every barrier.
	Then, for $\varepsilon \leq  \gamma r_\delta$,
	\begin{align}
		\bbP\bigl(X_{ r_\delta^2} \in\sB(y,\varepsilon) \mid X_0 = x_0 \bigr) \geq c(\varepsilon/ r_\delta )^2.\label{eq:VividQuip}
	\end{align}
\end{lemma}

\begin{proof}
	Let $\gamma$ and $\delta_0$ be sufficiently small so that Lemma \ref{lem: StripIntersect} is applicable.
	Then, by Lemma \ref{lem: PillShaped} with $R =  r_\delta$ there exists an absolute constant $c$ with
	\begin{equation}
		\bbP\bigl(x_0 + W_{ r_\delta^2} \in \sB(y,\varepsilon) \text{ and } \Vert W_t - (t/ r_\delta^2)(y - x_0)\Vert \leq  \gamma r_\delta, \forall t\leq  r_\delta^2 \bigr) \geq c(\varepsilon/ r_\delta )^2.\label{eq:MoistIce}
	\end{equation}
	The second condition in \eqref{eq:MoistIce} ensures that $x_0 + W_t$ is always at distance $\leq  \gamma r_\delta$ from a convex combination of $x_0$ and $y$.
	Hence, Lemma \ref{lem: StripIntersect} implies that $x_0 + W_t$ does not hit $\cup_{i=0}^m B_i$ for $t\leq  r_\delta^2$.
	The stochastic differential equation in Definition \ref{def: ReflectedBrownianMotion} then implies that $X_t = x_0 + W_t$ since the local time only increases at times when a barrier is hit.
	The first condition in \eqref{eq:MoistIce} can then be equivalently stated as $X_{ r_\delta^2} \in \sB(y,\varepsilon)$ given $X_0 = x_0$.
	This implies \eqref{eq:VividQuip}.
\end{proof}

\begin{figure}[hbt]
	\centering
	\includegraphics[width = 0.9\textwidth]{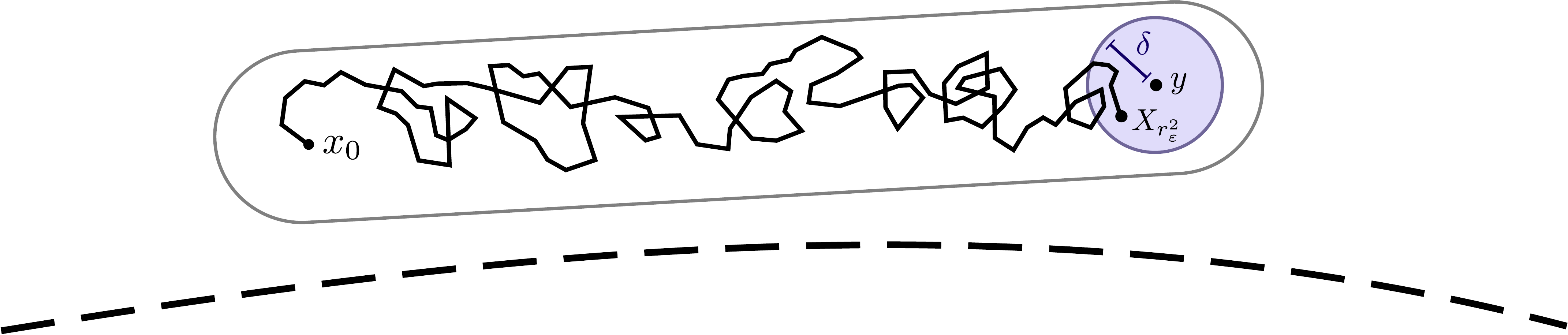}
	\caption{
		Visualization of the proof idea for Lemma \ref{lem: TransportSameSide}.
		By using Lemma \ref{lem: PillShaped} on a sufficiently narrow neighborhood of the line connecting $x_0$ to $y$, we can ensure that the manufactured process $x_0 + W_t$ never hits any barrier.
		Then, $X_{r_\delta^2} = x_0 + W_{r_\delta^2}$ lands in the ball of radius $\varepsilon$ around $y$.
	}
	\label{fig: AvoidBarrier}
\end{figure}

\subsubsection{Combining the estimates}
The combination of Lemma \ref{lem: ReductionSameSide} and Lemma \ref{lem: TransportSameSide} yields an estimate for general $x_0,y \in \cD(\delta)$ at distance $\leq (4/5) r_\delta$.
\begin{corollary}\label{cor: TransportClose}
	There exist absolute constants $\delta_0,\gamma,c>0$ such that the following holds for every $\delta \leq \delta_0$.
	Consider $x_0,y \in \cD(\delta)$ with $\Vert x_0-y \Vert\leq (4/5) r_\delta$.
	Then, for every $\varepsilon \leq \gamma r_\delta$,
	\begin{align}
		\bbP\bigl(X_{t_\delta} \in\sB(y,\varepsilon) \mid X_0 = x_0 \bigr) \geq c r_\delta \lambda_{\min} (\varepsilon/ r_\delta)^2\ \text{ with }\ t_\delta \de 2 r_\delta^2.\label{eq:PinkMob}
	\end{align}
\end{corollary}
\begin{proof}
	If $x_0$ and $y$ are on different sides of the barrier, then one can first apply Lemma \ref{lem: ReductionSameSide} to ensure that $X_{ r_\delta^2}$ is on the same side as $y$ with probability $\geq c_1  r_\delta \lambda_{\min}$ and subsequently apply Lemma \ref{lem: TransportSameSide} to ensure that $X_{2 r_\delta^2} \in \sB(y,\varepsilon)$ with (conditional) probability $\geq c_2 (\varepsilon/  r_\delta)^2$.
	The combination then yields \eqref{eq:PinkMob} with absolute constant $c_1 c_2$.

	If $x_0$ and $y$ are on the same side of the barrier,   then one can first burn a bit of time to keep the timing consistent with the other case.
	Specifically,  using Lemma \ref{lem: TransportSameSide} with $\varepsilon$ replaced by $\gamma r_\delta$ and $\gamma$ sufficiently small, it may be ensured that $X_{r_\delta^2}$ remains on the same side and satisfies the same conditions as $x_0$ does in the assumption of Corollary \ref{cor: TransportClose} with probability $\geq c_3$.
	Hereafter, Lemma \ref{lem: TransportSameSide} yields \eqref{eq:PinkMob} with absolute constant $c_3 c_2$ provided that one also uses that $ r_\delta \lambda_{\min} \leq \delta \leq 1$.
	Take $c \de \min\{c_1c_2, c_3c_2 \}$ to conclude.
\end{proof}

\subsection{Points distant from the barriers and from each other}\label{sec: DistantEndpoints}

We next consider the case where $x_0,y \in \cD(\delta)$ are not necessarily near to each other.
The idea is then to iteratively use Corollary \ref{cor: TransportClose} to replace $x_0$ by some point with reduced distance to $y$.
In this context, the suitable notion of distance is not the Euclidean distance but rather the \emph{geodesic distance}:
\begin{align}
	\dgeo(x_0,y) \de \inf_{\gamma} \int_0^1 \Vert \gamma'(t) \Vert \intd t.\label{eq: Def_dgeo}
\end{align}
The infimum here runs over piecewise smooth curves $\gamma:[0,1]\to D$ with $\gamma(0) = x_0$ and $\gamma(1) =y$.
A visualization of the approach may be found in Figure \ref{fig: Geo}.

\begin{figure}[bht]
	\centering
	\includegraphics[width = 0.85\textwidth]{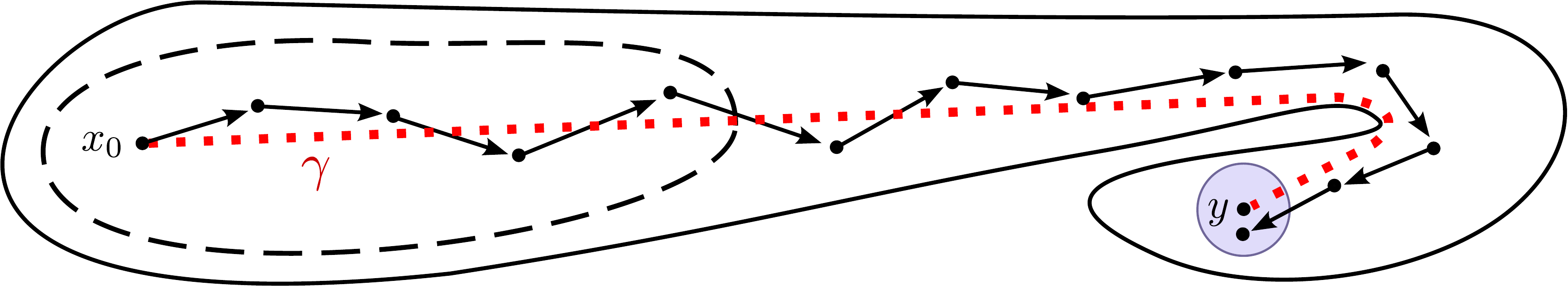}
	\caption{
		Visualization of the approach in Section \ref{sec: DistantEndpoints}.
		A geodesic $\gamma$ connecting $x_0$ to $y$ is visualized by the red dotted line.
		The main idea is that Corollary \ref{cor: TransportClose} allows one to iteratively transport some probability mass along this curve.
	}
	\label{fig: Geo}
\end{figure}

We start with some preparatory lemmas.
We need these because Corollary \ref{cor: TransportClose} only applies when both endpoints are at nontrivial distance from the barriers, while a geodesic may get arbitrarily close to the barriers.

\begin{lemma}\label{lem: ClosePoint}
	For $\mu,\nu>0$ with $1-\mu -\nu \geq 1/5$ and $\mu -\nu > 1/5$ there exists some $\delta_0 \geq 0$ such that the following holds for $\delta \leq \delta_0$.
	For every $x\in D$ there exists some $z \in \cD(\delta)$ with $\Vert x- z \Vert \leq \mu r_\delta$ such that $\sB(z,  \nu r_\delta)\subseteq \cD(\delta)$ and $\{\alpha z + (1-\alpha)x: \alpha \in [0,1] \} \subseteq D$.
\end{lemma}

\begin{proof}
	If the intersection of $\sB(x_0, r_\delta)$ with $\cup_{i=0}^m B_i$ is empty, then we can take $z = x$.
	Assume that there is some $i\geq 0$ with $B_i\cap \sB(x,  r_\delta) \neq \emptyset$.
	Then, we show that the following works:
	\begin{align}
		z \de
		\begin{cases}
			x +  \mu r_\delta \hatn & \text{ if }x \text{ is on the positive side of }B_i, \\
			x - \mu r_\delta \hatn  & \text{ otherwise}.
		\end{cases}\label{eq:GladJet}
	\end{align}
	Here, recall that $x$ is on the positive side of $B_i$ when $x\in B_i$.
	This is important if $i=0$ as $x - \mu r_\delta\hatn$ may not even be an element of $D$ when $x\in B_0$.

	It makes no substantial difference for the subsequent arguments whether $x$ is on the positive or negative side if $i\neq 0$.
	Further, $x\in D$ is always on the positive side of $B_0$.
	To simplify notation, let us hence focus on the case where $x$ is on the positive side of $B_i$.

	The proposed $z$ certainly satisfies $\Vert x - z \Vert \leq  \mu r_\delta$.
	We next verify that $\sB(z, \nu r_\delta) \subseteq \cD(\delta)$.
	Lemma \ref{lem: BarrierLocallyStraight} yields that $B_i$ is the only barrier which intersects $\sB(x,  r_\delta)$.
	In particular, every point in $\sB(z,  \nu r_\delta)$ has distance $> (1- \mu - \nu) r_\delta \geq  r_\delta/5$ from all $B_j$ with $j\neq i$.
	Similarly, the distance from $B_i\setminus \sB(x,r_\delta)$ is $\geq r_\delta/5$.
	Recalling the definition of $\cD(\delta)$ from \eqref{eq: Def_Deps}, we hence have to show that every point in $\sB(z,  \nu r_\delta)$  lies in $D$ and has distance $\geq  r_\delta/5$ from $B_i \cap \sB(x, r_\delta)$.

	Recall that we are focusing on the case where $x$ is on the positive side of $B$.
	Consequently, by \ref{item:JollyDragon} in Lemma \ref{lem: BarrierLocallyStraight}, we have
	$
		\langle x,  \hatn  \rangle \geq  \mathfrak{c} -4 \delta  r_\delta$.
	Hence, for $\zeta \in \sB(z,\nu r_\delta)$,
	\begin{align}
		\langle \zeta , \hatn \rangle \geq \langle x , \hatn \rangle + (\mu  - \nu) r_\delta \geq \mathfrak{c} + (\mu  - \nu - 4\delta) r_\delta.
	\end{align}
	On the other hand, recall from item \ref{item:GhostlyQuip} that for $y \in B_i \cap \sB(x_0,r_\delta)$, $\langle y, \hatn \rangle \leq \mathfrak{c} + 4\delta r_\delta$.
	Then, by the Cauchy--Schwarz inequality,
	\begin{align}
		\Vert \zeta -y  \Vert \geq \langle \zeta - y , \hatn \rangle  \geq (\mu - \nu - 8 \delta r_\delta) \ \text{ for every }y \in B_i \cap \sB(x_0,r_\delta).
	\end{align}
	Recall that $\mu - \nu > 1/5$.
	Hence, taking $\delta_0$ sufficiently small, it may be ensured that $\mu  - \nu - 8\delta > 1/5$.
	This proves that $\sB(z,\nu r_\delta) \subseteq \cD(\delta)$ as desired.

	Let us finally verify that all convex combinations of $x$ and $z$ are in $D$.
	Suppose to the contrary that $Z(\alpha) \de \alpha z + (1-\alpha) x$ is not in $D$ for some $\alpha \in [0,1]$ and set
	\begin{align}
		\alpha_* \de \inf\{\alpha \in [0,1] : Z(\alpha) \not\in D \}.
	\end{align}
	Then, $Z(\alpha_*) \in B_0$ and $Z(\alpha_* + \delta)$ is strictly on the negative side of $B_0$ for $\delta>0$ sufficiently small.
	Then, recalling that $\vecn_0$ points towards the positive side, the derivative of $Z$ satisfies
	$
		\bigl\langle Z'(\alpha_*), \vecn_0(Z(\alpha_*)) \bigr\rangle \leq 0.
	$
	But, by definition of $Z(\alpha)$ and by \eqref{eq:GladJet}, we have $Z'(\alpha_*) = \mu r_\delta \hatn$, and it follows from item \ref{item:IdleBat} in Lemma \ref{lem: BarrierLocallyStraight} that $\langle \hatn, \vecn_0(Z(\alpha_*)) \rangle \geq 1 - 2\delta > 0$.
	This yields a contradiction.
\end{proof}
\begin{lemma}\label{lem: ReduceGeo}
	There exist absolute constants $\delta_0, \nu >0$ such that the following holds for every $\delta \leq \delta_0$.
	Consider $x_0,y \in D$ with $\dgeo(x_0, y) \geq (4/5) r_\delta$.
	Then, there exists $z\in \cD(\delta)$ with $\Vert z - x_0 \Vert \leq (4/5) r_\delta$ satisfying that $\sB(z,\nu r_\delta)\subseteq \cD(\delta)$ and
	\begin{align}
		\dgeo(\zeta, y) \leq \dgeo(x_0, y) -  r_\delta/10 \ \text{ for every } \ \zeta \in \sB(z,  \mu r_\delta).  \label{eq:FastWok}
	\end{align}
\end{lemma}
\begin{proof}
	By definition of the geodesic distance, there exists a piecewise smooth curve $\gamma:[0,1]\to D$ with length $\leq \dgeo(x_0, y) +  r_\delta/100$ satisfying $\gamma(0) = x_0$ and $\gamma(1) = y$.
	Moreover, $\gamma$ then necessarily has length $\geq \dgeo(x_0, y) \geq (4/5) r_\delta$.
	Denote $\tau$ for the least time such that the piece of $\gamma$ connecting $x_0$ to $\gamma(\tau)$ has length $\geq r_\delta/2$:
	\begin{align}
		\tau \de \inf\Bigl\{t\in [0,1]: \int_0^t \Vert \gamma'(s) \Vert\, \intd s \geq  r_\delta/2 \Bigr\}.\label{eq: ProofGeo_Deftau}
	\end{align}
	We next apply Lemma \ref{lem: ClosePoint} with suitable parameters such as, say, $\mu = 3/10$ and $\nu = 1/100$.
	This yields some $z \in \cD(\delta)$ with $\Vert z - \gamma(\tau) \Vert \leq (3/10) r_\delta$ satisfying that $\sB(z,  \nu r_\delta) \subseteq \cD(\delta)$ and $\{\alpha z + (1-\alpha)\gamma(\tau) : \alpha \in [0,1]\}\subseteq D$.

	By definition of $\tau$, the piece of $\gamma$ connecting $x_0$ to $\gamma(\tau)$ has length $r_\delta/2$.
	Hence, since the shortest curve connecting two points in $\bbR^2$ is a straight line, $\Vert x_0 -  \gamma(\tau) \Vert \leq  r_\delta/2$.
	By the triangle inequality, it then follows that $\Vert z - x_0 \Vert \leq (3/10 + 1/2) r_\delta = (4/5) r_\delta$.
	What remains to be shown is that \eqref{eq:FastWok} holds.

	Fix some $\zeta \in \sB(z,  \nu r_\delta)$ and define a piecewise smooth $\tilde{\gamma}: [0,1] \to \bbR^2$ satisfying $\tilde{\gamma}(0) = \zeta$ and $\tilde{\gamma}(1) = y$ by first going in a straight line from $\zeta$ to $z$, then going in a straight line from $z$ to $\gamma(\tau)$, and finally following $\gamma$ towards $y$:
	\begin{align}
		\tilde{\gamma}(t) =
		\begin{cases}
			\zeta + (2t/\tau)\bigl(z - \zeta) \qquad                   & \text{ if }t \leq \tau/2,        \\
			z + ((2t - \tau)/\tau)\bigl(\gamma(\tau) - z\bigr), \qquad & \text{ if } \tau/2 < t\leq \tau, \\
			\gamma(t), \qquad                                          & \text{ if }t >\tau.
		\end{cases} \label{eq:HappyZebra}
	\end{align}
	Observe that $\{\tilde{\gamma}(t): t\leq \tau/2 \}\subseteq D$ because $\sB(z, \nu r_\delta)\subseteq \cD(\delta)$ and $\{\tilde{\gamma}(t): \tau/2 \leq t\leq \tau \}\subseteq D$ since $\{\alpha z + (1-\alpha)\gamma(\tau) : \alpha \in [0,1]\}\subseteq D$ by definition of $z$.
	Hence, also using that $\{\tilde{\gamma}(t): \tau <t \leq 1 \}\subseteq D$ since $\gamma$ has codomain $D$, we can view $\tilde{\gamma}$ as a map from $[0,1]$ to $D$.
	Consequently, the length of $\tilde{\gamma}$ provides an upper bound on $\dgeo(\zeta, y)$.

	Considering the contributions to the length for $t\leq \tau/2$, $\tau/2 < t\leq \tau$, and $t>\tau$ separately, we have
	\begin{align}
		\int_0^1 \Vert \tilde{\gamma}'(t) \Vert \, \intd t & =  \Vert \zeta - z \Vert + \Vert z - \gamma(\tau) \Vert +  \int_\tau^1 \Vert \gamma'(t) \Vert \, \intd t.   \label{eq:ElegantLion}
	\end{align}
	Recall that, by definition of $\tau$, the piece of $\gamma$ connecting $x_0$ to $\gamma(\tau)$ has length $r_\delta/2$.
	Further, recall that $\gamma$ was chosen to have total length $\leq \dgeo(x_0, y) +  r_\delta/100$.
	Consequently,
	\begin{align}
		\int_\tau^1 \Vert \gamma'(t) \Vert \, \intd t & = \int_0^1 \Vert \gamma'(t) \Vert \, \intd t - \int_0^\tau \Vert \gamma'(t) \Vert \, \intd t  \leq  \dgeo(x_0, y) +  r_\delta/100 -  r_\delta/2.\label{eq:CleverUser}
	\end{align}
	Combine \eqref{eq:ElegantLion} and \eqref{eq:CleverUser} with the facts that $\Vert \zeta - z \Vert \leq \nu r_\delta =  r_\delta /100$ and $\Vert z - \gamma(\tau) \Vert  \leq (3/10) r_\delta$ to conclude that
	\begin{align}
		\int_0^1 \Vert \tilde{\gamma}'(t) \Vert\, \intd t & \leq \dgeo(x_0, y) + (3/10 + 2/100- 1/2)  r_\delta \leq \dgeo(x_0, y) -  r_\delta/10.
	\end{align}
	This proves \eqref{eq:FastWok}.
\end{proof}
It is convenient to also have a result for the case with $\dgeo(x_0,y) \leq (4/5)r_\delta$.
In this case we do not try to decrease the geodesic distance, but simply avoid enlarging it:

\begin{lemma}\label{lem: DoNotGrowGeo}
	There exist absolute constants $\delta_0, \nu >0$ such that the following holds for every $\delta \leq \delta_0$.
	Consider $x_0,y \in D$ with $\dgeo(x_0, y) \leq (4/5) r_\delta$.
	Then, there exists $z\in \cD(\delta)$ with $\Vert z - x_0 \Vert \leq (4/5) r_\delta$ satisfying $\sB(z,\nu r_\delta)\subseteq \cD(\delta)$ and
	\begin{align}
		\dgeo(\zeta, y) \leq (4/5)  r_\delta \ \text{ for }\ \zeta \in \sB(z, \nu r_\delta). \label{eq:CrimsonEgg}
	\end{align}
\end{lemma}
\begin{proof}
	This proof proceeds similar to that of Lemma \ref{lem: ReduceGeo}, except that we modify the definition of $\tau$.
	Take $\gamma:[0,1]\to D$ to be a piecewise smooth curve with $\gamma(0) = x_0$ and $\gamma(1) = y$ whose length is at most $\dgeo(x_0, y) +  r_\delta/100$.
	Let
	\begin{align}
		\cT \de \inf\Bigl\{t\in [0,1] : \int_0^t \Vert \gamma'(s) \Vert \, \intd s = \int_t^1 \Vert \gamma'(s) \Vert\intd s \Bigr\}.\label{eq:ValidHen}
	\end{align}
	Then, proceeding identically to \eqref{eq: ProofGeo_Deftau}--\eqref{eq:ElegantLion} one can find some $z\in \cD(\delta)$ with $\Vert z - \gamma(\cT) \Vert \leq (3/10)r_\delta$ and $\sB(z,\nu r_\delta) \subseteq \cD(\delta)$ such that every $\zeta \in \sB(z,\nu r_\delta)$ admits a curve $\tilde{\gamma}:[0,1]\to D$ with $\tilde{\gamma}(0) = \zeta$ and $\tilde{\gamma}(1) = z$ satisfying $\int_0^1 \Vert \tilde{\gamma}'(t) \Vert \, \intd t =  \Vert \zeta - z \Vert + \Vert z - \gamma(\cT) \Vert +  \int_\cT^1 \Vert \gamma'(t) \Vert \, \intd t.$
	Here, recalling \eqref{eq:ValidHen} and the assumed upper bounds on the length of $\gamma$,
	\begin{align}
		\int_\cT^1 \Vert \gamma'(t) \Vert \, \intd t = \frac{1}{2}\int_0^1 \Vert \gamma'(t) \Vert \, \intd t \leq \frac{1}{2}\bigl( \dgeo(x_0, y) + r_\delta /100\bigr).
	\end{align}
	Using the assumptions regarding $\dgeo(x_0,y)$ and a direct numerical computation now yields that $\int_0^1 \Vert \tilde{\gamma}'(t) \Vert\, \intd t \leq (4/5)r_\delta$.
	This implies that \eqref{eq:CrimsonEgg} is satisfied.

	The desired bound on $\Vert z-x_0 \Vert$ can be deduced similarly.
	Note that $\Vert z-x_0 \Vert \leq \Vert z - \gamma(\cT) \Vert + \Vert \gamma(\cT) - x_0 \Vert$.
	The first term was bounded by $(3/10)r_\delta$ in the discussion after \eqref{eq:ValidHen}, and the second term is bounded by $\int_0^\cT \Vert \gamma'(t) \Vert \intd t \leq (\dgeo(x_0,y) + r_\delta/100)/2 \leq (2/5 + 1/200)r_\delta$.
	Hence, using that $3/10 + 2/5 + 1/200 \leq 4/5$ yields that $\Vert z-x_0 \Vert \leq (4/5)r_\delta$, as desired.
\end{proof}

Recall from \eqref{eq: Def_Delta} that the parameter $\Delta$ provides an upper bound on the geodesic distance between any two points $x_0,y \in D$.
We can now iteratively decrease the geodesic distance to get the desired estimate:

\begin{corollary}
	\label{cor: TransportDistant}
	There exist absolute constants $\delta_0, \gamma,c,C>0$ such that the following holds for $\delta \leq \delta_0$ and $\varepsilon \leq \gamma r_\delta$.
	Let $T \de (2\lceil 10 \Delta /  r_\delta \rceil + 2) r_\delta^2$.
	Then, for $x_0, y \in \cD(\delta)$,
	\begin{align}
		\bbP(X_T \in \sB(y, \varepsilon) \mid X_0 = x_0) \geq (c  r_\delta \lambda_{\min})^{C\Delta/ r_\delta} (\varepsilon/ r_\delta)^2.\label{eq:UneasySoup}
	\end{align}
\end{corollary}

\begin{proof}
	For $d\geq 0$ and $t>0$, set
	\begin{equation}
		P(t,d) \de
		\inf\bigl\{\bbP\bigl(X_t \in \sB(y,\varepsilon)\mid X_0 = x_0 \bigr): x_0,y \in \cD(\delta), \dgeo(x_0,y) \leq d \bigr\}.\label{eq:DrySwan}
	\end{equation}
	If $\dgeo(x_0,y) \geq (4/5)r_\delta$, then combining Lemma \ref{lem: ReduceGeo} and Corollary \ref{cor: TransportClose} implies that $X_{2r_\delta^2}$ will reduce the geodesic distance to $y$ by at least $r_\delta/10$ with probability $\geq c r_\delta \lambda_{\min}$ for some absolute constant $c>0$.
	Similarly, if $\dgeo(x_0,y) \leq (4/5)r_\delta$, then combining Lemma \ref{lem: DoNotGrowGeo} and Corollary \ref{cor: TransportClose} yields that $X_{2r_\delta^2}$ is again at geodesic distance $\leq (4/5)r_\delta$ from $y$ with probability $\geq c r_\delta \lambda_{\min}$.
	Hence, by the law of total probability, it holds for $d\geq 0$ and $t\geq 2r_\delta^2$ that
	\begin{align}
		P(t,d)
		\geq
		c_1 r_\delta \lambda_{\min} P\bigl(t- 2 r_\delta^2, \max\{d- r_\delta/10, (4/5)r_\delta \}\bigr).
		\label{eq:SilentTin}
	\end{align}
	Recall the definition of $T$ and iteratively apply \eqref{eq:SilentTin} to find that
	\begin{align}
		P(T, \Delta) \geq (c  r_\delta \lambda_{\min})^{\lceil 10\Delta / r_{\delta} \rceil } P\bigl(2 r_\delta^2, (4/5) r_\delta \bigr).\label{eq:JumpyRoad}
	\end{align}
	Further, using Corollary \ref{cor: TransportClose} and the fact that $\Vert x_0 - y \Vert \leq (4/5)r_\delta$ whenever $\dgeo(x_0,y) \leq (4/5)r_\delta$, it holds that $P\bigl(2 r_\delta^2, (4/5) r_\delta \bigr) \geq c  r_\delta \lambda_{\min} (\varepsilon /  r_\delta)^2$.

	It can be shown using Lemma \ref{lem: BarrierLocallyStraight} that the diameter of $D$ is at least $r_\delta$ if $\delta$ is taken sufficiently small.
	Consequently, $\lceil 10 \Delta/ r_\delta \rceil + 1 \leq C \Delta/r_\delta$ for some sufficiently large constant $C>0$.
	Then, also using that $ r_\delta \lambda_{\min}\leq \delta\lambda_{\min}/\lambda_{\max} \leq 1$,
	\begin{align}
		P(T, \Delta)\geq (c  r_\delta \lambda_{\min})^{\lceil 10\Delta / r_{\delta} \rceil +1}(\varepsilon/ r_\delta)^2 \geq  (c  r_\delta \lambda_{\min})^{C \Delta / r_{\delta} } (\varepsilon/ r_\delta)^2.\label{eq:NormalLeaf}
	\end{align}
	The fact that $\dgeo(x_0,y) \leq \Delta$ for all $x_0,y\in D$ now yields \eqref{eq:UneasySoup}.
\end{proof}

\subsection{Points which could be close to the barriers}\label{sec: ReductionClose}
Finally, we consider the case where $x_0$ or $y$ is potentially close to a barrier.
We start by a reduction argument showing that the case where $x_0$ is close to a barrier follows from the previously established estimates with $x_0$ distant from the barrier.
This reduction exploits \cite[Lemma 4.8]{vanwerde2024recovering} which states that $X_t$ is typically at distance of order $\geq \sqrt{t}$ from the barriers for small $t$:

\begin{lemma}\label{lem: x0_reduction_distant}
	There exist absolute constants $\delta_0,c>0$ such that the following holds for $\delta \leq \delta_0$.
	For every $x_0\in D$, measurable $E\subseteq D$, and $t \geq c\min\{ 1/\kappa^2, 1/\lambda_{\max}^2,\rho^2 \}$,
	\begin{equation}
		\bbP\bigl(X_t \in E \mid X_0 = x_0\bigr)
		\geq
		\frac{1}{2}
		\sup_{u}\inf_{x_0' \in \cD(\delta)}\bbP\bigl(X_{u} \in E \mid X_0 = x_0'\bigr)
		.
		\label{eq:MadFog}
	\end{equation}
	Here, the supremum runs over $u \leq  t - c\min\{ 1/\kappa^2, 1/\lambda_{\max}^2,\rho^2\}$.
\end{lemma}
\begin{proof}
	The idea is to condition on some intermediate time and to exploit that the process will then be at a reasonable distance from every barrier, even if $x_0$ lies close to some barrier.
	More precisely, \cite[Lemma 4.8]{vanwerde2024recovering} yields absolute constants $c_1,c_2>0$ such that for every $t' \leq c_1\min\{1/\kappa^2, 1/\lambda_{\max}^2,\rho^2  \}$,
	\begin{align}
		\bbP\bigl(\forall y \in \cup_{i=0}^m B_i : \Vert X_{t'} -y\Vert \geq c_2 \sqrt{t'} \mid X_0 = x_0\bigr) \geq 1/2\ \text{ for every }\ x_0\in D.\label{eq:SilentQuote}
	\end{align}
	Recall the definitions of $r_\delta$ and $\cD(\delta)$ from \eqref{eq: Def_R_eps} and \eqref{eq: Def_Deps}.
	Hence, taking $\delta_0 \de 5 c_2 \sqrt{c_1}$, the event described in \eqref{eq:SilentQuote} means that $X_{t'} \in \cD(\delta_0)$.
	Then also $X_{t'} \in  \cD(\delta)$ for every $\delta \leq \delta_0$.
	From here on, let $c \de c_1$ and $t' \de c\min\{1/\kappa^2, 1/\lambda_{\max}^2,\rho^2  \}$.

	For every $u \leq t - t'$, the law of total probability yields that
	\begin{align}
		\bbP\bigl(X_t \in E \mid X_0 = x_0\bigr)  = \bbE\bigl[\bbE\bigl[\bbP\bigl( X_t \in E  \mid X_{t - u}\bigr)  \mid X_{t - t' - u}\bigr] \mid X_0 = x_0\bigr].
	\end{align}
	The desired result in \eqref{eq:MadFog} now follows by Markovianity since \eqref{eq:SilentQuote} and the subsequent discussion give that $X_{t-u}\in \cD(\delta)$ with probability $\geq 1/2$.
\end{proof}

The reduction for $y$ is more delicate.
In particular, unlike most of the previous arguments, a direct approach based on local approximations is not suitable here.
This problem stems from the fact that we desire an estimate on the probability to land in an \emph{arbitrarily small} target set.
This makes it easy for error terms in local approximations to exceed the size of the target, leading to trivial estimates.

Fortunately, a trick can be used circumvent this difficulty.
The following result allows us to effectively reverse the arrow of time, after which we can proceed similarly to Lemma \ref{lem: x0_reduction_distant}:
\begin{lemma}\label{lem: TrickTimeRev}
	Suppose that $X_0 \sim \Unif(D)$.
	Then, for every $t>0$ and every measurable $E_1,E_2 \subseteq D$ with nonzero measure,
	\begin{align}
		\bbP\Bigl(X_t \in E_2 {} & {}\mid X_0 \in E_1\Bigr)                                                                                                                       \\
		                         & \geq \frac{\Area(E_2)}{\Area(E_1)} \bbP\Bigl(X_t \in E_1, s_i(\Lc{i}_r) = s_i(0), \forall r \leq t, i\leq m  \mid X_0 \in E_2\Bigr). \nonumber
	\end{align}
\end{lemma}
\begin{proof}
	Denote $C_0$ for the closure of the component of $D\setminus (\cup_{i=0}^m B_i)$ associated with the initial conditions $X_0$ and $s_i(0)$.
	Let $Y_t$ be a classical reflected Brownian motion in $C_0$ with local time $\cL_t$ which has initial condition $X_0$ and is driven by the same Wiener process $W_t$.
	That is, the processes satisfying
	\begin{align}
		\intd Y_t = \intd W_t + \sum_{i=0}^m s_i(0) \vec{n}_i(Y_t) \bb1\{Y_t \in B_i \}\intd \cL_t\label{eq:NoisyFox}
	\end{align}
	subject to the usual characterizing conditions\footnote{That is, $Y_t$ is continuous with values in $C_0$, and $\cL_t$ is a continuous nondecreasing process with $\cL_0 = 0$ with values in $\bbR$ which only increases when $Y_t\in \partial C_0$.}.
	Note that, different from Definition \ref{def: ReflectedBrownianMotion}, the coefficients $s_i(0)$ in \eqref{eq:NoisyFox} are kept fixed and will not ever change sign.

	Observe that $X_t = Y_t$ and $\Lc{i}_t = \int_0^t \bb1\{Y_r \in B_i \}\intd\cL_r$ for all $t$ that satisfy $s_i(\Lc{i}_r) = s_i(0)$ for every $r \leq t$.
	In particular,
	\begin{align}
		\bbP\Bigl(X_t \in{} & {} E_2, X_0 \in E_1, s_i(\Lc{i}_r) = s_i(0), \forall r \leq t, i\leq m\Bigr)\label{eq:FairRiver}                                                     \\
		                    & = \bbP\Bigl(Y_t \in E_2, Y_0 \in E_1, s_i({\textstyle \int_0^r}\bb1\{Y_z \in B_i \} \intd \cL_z) = s_i(0), \forall r \leq t, i\leq m\Bigr).\nonumber
	\end{align}
	The key idea, which we next explain in detail, is to exploit the time reversibility of classical reflected Brownian motion to rewrite the right-hand side of \eqref{eq:FairRiver}.

	Define $\tiY_r \de Y_{t-r}$ and $\tilde{\cL}_r \de \cL_t - \cL_{t-r}$ for $r \leq t$.
	Then, by definition,
	\begin{align}
		\bbP\Bigl(Y_t \in {} & {} E_2, Y_0 \in E_1, s_i({\textstyle \int_0^r}\bb1\{Y_z \in B_i \} \intd \cL_z) = s_i(0), \forall r \leq t, i\leq m\Bigr)\label{eq:OccultZebra}                      \\
		                     & =\bbP\Bigl(\tiY_0 \in E_2, \tiY_t \in E_1, s_i({\textstyle \int_0^r}\bb1\{\tiY_z \in B_i \} \intd \tilde{\cL}_z) = s_i(0), \forall r \leq t, i\leq m\Bigr).\nonumber
	\end{align}
	Remark that the process defined by $\tiW_r \de W_t - W_{t-r}$ is again a Wiener process and note that \eqref{eq:NoisyFox} and the associated conditions imply that $\tiY_r$ satisfies
	\begin{align}
		\intd \tiY_r = \intd \tiW_r + \sum_{i=0}^m s_i(0)\vec{n}_i(\tiY_r) \bb1\{\tiY_r \in B_i \}\intd \tilde{\cL}_r\label{eq:JollyUrn}
	\end{align}
	subject to the usual conditions.

	In other words, $\tiY_r$ is a classical reflected Brownian motion with initial condition $\tiY_0 = Y_t$ and local time $\tilde{\cL}_r$ in the random component $C_0$.
	Since the uniform distribution is the stationary distribution for classical reflected Brownian motion, the random variable $\tiY_0 = Y_t$ conditional on the component $C_0$ is again uniform on that component.
	Hence, averaging out the dependence on the random component, the random variable $\tiY_0$ is uniform on $D$.

	So, the processes $Y$ and $\tiY$ are both classical reflected Brownian motions in a random component, started from the uniform distribution on $D$.
	It follows that the processes and their local times have the same law.
	In particular,
	\begin{align}
		\bbP\Bigl(\tiY_0 \in {} & {}E_2, \tiY_t \in E_1, s_i({\textstyle \int_0^r}\bb1\{\tiY_z \in B_i \} \intd \tilde{\cL}_z) = s_i(0), \forall r \leq t, i\leq m\Bigr)\label{eq:SlowMob} \\
		                        & = \bbP\Bigl(Y_0 \in E_2, Y_t \in E_1, s_i({\textstyle \int_0^r}\bb1\{Y_z \in B_i \} \intd \cL_z) = s_i(0), \forall r \leq t, i\leq m\Bigr).\nonumber
	\end{align}
	Now, by \eqref{eq:FairRiver}--\eqref{eq:SlowMob} and the definition of conditional probability,
	\begin{align}
		\bbP\Bigl(X_t \in E_2, {} & {} s_i(\Lc{i}_r) = s_i(0), \forall r \leq t, i\leq k \mid X_0 \in E_1\Bigr) \bbP\Bigl(X_0 \in E_1\Bigr)\label{eq:RedKey}                \\
		                          & = \bbP\Bigl(X_t \in E_1,  s_i(\Lc{i}_r) = s_i(0), \forall r \leq t, i\leq k\mid X_0 \in E_2\Bigr)\bbP\Bigl(X_0 \in E_2\Bigr). \nonumber
	\end{align}
	Here, by the monotonicity of probability,
	\begin{equation}
		\bbP\Bigl(X_t \in E_2 \mid X_0 \in E_1\Bigr) \geq \bbP\Bigl(X_t \in E_2, s_i(\Lc{i}_r) = s_i(0), \forall r \leq t, \forall i\leq k \mid X_0 \in E_1\Bigr). \label{eq:UsefulWisp}
	\end{equation}
	Combine \eqref{eq:RedKey} with \eqref{eq:UsefulWisp} and use that $\bbP(X_0 \in E) = \Area(E)/\Area(D)$ by $X_0$ being uniform on $D$ to conclude the proof.
\end{proof}

\begin{corollary}\label{cor: yreduction}
	For every $\delta_0 >0$ there exist $\delta \leq \delta_0$ and $c>0$ such that the following holds for every $x_0\in D$.
	Suppose that one is given $T,C,\gamma>0$ such that for every $\varepsilon \leq \gamma r_\delta$,
	\begin{align}
		\bbP(X_T \in \sB(y,\varepsilon) \mid X_0 = x_0)\geq C \Area(\sB(y,\varepsilon))  \text{ for every } y \in \cD(\delta). \label{eq:BoldCamel}
	\end{align}
	Then, for $T' \de T + c \min\{1/\kappa^2, 1/\lambda_{\max}^2, \rho^2 \}$ and every measurable $E\subseteq D$,
	\begin{align}
		\bbP(X_{T'} \in E \mid X_0 = x_0)\geq (C/2)\Area(E).\label{eq:WittyDog}
	\end{align}
\end{corollary}
\begin{proof}
	Let $Q(\cdot) \de \Unif(\cD(\delta))$ be the uniform probability measure on $\cD(\delta)$.
	Standard measure-theoretic arguments\footnote{Using twice that Borel probability measures are outer-regular, it suffices to prove that \eqref{eq:TediousFox} holds for open sets.
		Further, if $A\subseteq \cD(\varepsilon)$ is an open set, then the Vitali covering theorem implies that there exists a countable disjoint collection of balls $B_i \subseteq A$ with $\Area(A \setminus \cup_{i=1}^\infty B_i) = 0$.
		Applying \eqref{eq:BoldCamel} to the $B_i$ then yields \eqref{eq:TediousFox} for the open set $A$.} imply that \eqref{eq:BoldCamel} being valid for open balls implies that a similar inequality holds for general measurable sets:
	\begin{equation}
		\bbP(X_T \in A  \mid X_0 = x_0) \geq C\Area(\cD(\delta)) Q(A),\ \forall\, \text{measurable }\, A\subseteq \cD(\delta). \label{eq:TediousFox}
	\end{equation}
	It was here used that $Q(A) = \Area(A)/\Area(\cD(\delta))$, by definition.
	It follows from \eqref{eq:TediousFox} that for every measurable function $f:\cD(\delta) \to \bbR_{\geq 0}$,
	\begin{align}
		\bbE\bigl[f(X_T) \mid X_0 = x_0\bigr] \geq C \Area(\cD(\delta)) \bbE\bigl[f(Y)\mid Y \sim \Unif(\cD(\delta)) \bigr].  \label{eq:NormalCobra}
	\end{align}
	Let $f(x) \de \bbP(X_{T' - T}\in E \mid X_0 = x )$ and note that we can alternatively write $f(x) = \bbP(X_{T'}\in E \mid X_T = x )$.
	Substituting the definition of $f$ on the right-hand side of \eqref{eq:NormalCobra} and the alternative expression on the left-hand side, and using the law of total probability to simplify the left-hand side, we find
	\begin{equation}
		\bbP(X_{T'} \in E \mid X_0 = x_0)  \geq C\Area(\cD(\delta)) \bbP\bigl(X_{T' - T} \in E \mid X_0 \sim \Unif(\cD(\delta))\bigr).\nonumber
	\end{equation}
	Consequently, using Lemma \ref{lem: TrickTimeRev} with $E_1 \de \cD(\delta)$ and $E_2 \de E$,
	\begin{align}
		\bbP{} & {}(X_{T'} \in  E \mid X_0 = x_0) \label{eq:RoyalFan}                                                                                     \\
		       & \geq C \Area(E) \bbP(X_{T' - T} \in \cD(\delta),\, s_i(\Lc{i}_r) = s_i(0),\, \forall r \leq t, i\leq m \mid X_0 \sim \Unif(E)).\nonumber
	\end{align}
	It remains to show that the probability occurring on the right-hand side of \eqref{eq:RoyalFan} can be made $\geq 1/2$ by choosing $\delta$ and $c$ appropriately.

	Indeed, exactly as in \eqref{eq:SilentQuote}, it follows from \cite[Lemma 4.8]{vanwerde2024recovering} that there exist absolute constants $c_1,c_2>0$ such that for every $t \leq c_1\min\{1/\kappa^2, 1/\lambda_{\max}^2,\rho^2  \}$,
	\begin{align}
		\bbP\bigl(\forall y \in \cup_{i=0}^m B_i : \Vert X_{t} -y\Vert \geq c_2 \sqrt{t} \mid X_0 = x_0\bigr) \geq 1 -  1/4,\ \forall x_0\in D.\label{eq:SilentQuote2}
	\end{align}
	In particular, recalling the definition of $\cD(\delta)$ from \eqref{eq: Def_Deps}, it holds on the event in \eqref{eq:SilentQuote2} that $X_t\in \cD(\delta)$ if $\delta$ is such that $r_\delta /5 \leq  c_2 \sqrt{t}$.
	Further, \cite[Lemma 4.7]{vanwerde2024recovering} yields an absolute constant $c_3>0$ such that for every $t \leq c_3\min\{1/\kappa^2, 1/\lambda_{\max}^2,\rho^2 \}$,
	\begin{align}
		\bbP\bigl(s_i(\Lc{i}_t) = s_i(0), \, \forall i \leq m  \mid X_0 = x_0\bigr) \geq 1 - 1/4,\ \forall x_0 \in D. \label{eq:PinkGym}
	\end{align}
	Let $c \de \min\{c_1,c_3 \}$ and let $\delta \de 5 c_2 \sqrt{c}$.
	Then, recalling that $T' = T + c\min\{1/\kappa^2, 1/\lambda_{\max}^2, \rho^2 \}$ and combining \eqref{eq:SilentQuote2}--\eqref{eq:PinkGym},
	\begin{align}
		\bbP(X_{T' - T} \in \cD(\delta),\, s_i(\Lc{i}_r) = s_i(0),\, \forall r \leq t, i\leq m \mid X_0 \sim \Unif(E)) \geq 1/2. \label{eq:ZippyYawn}
	\end{align}
	Combine \eqref{eq:RoyalFan} and \eqref{eq:ZippyYawn} to conclude that \eqref{eq:WittyDog} holds, as desired.
\end{proof}
It remains to combine the foregoing estimates and to simplify the results into a more readable form.
This amounts to straightforward but notationally cumbersome computations.
\begin{corollary}\label{cor: DoeblinHolds}
	There exist absolute constants $c_1,c_2>0$ such that the following holds.
	Denote $R \de c_1\min\{1/\kappa, 1/\lambda_{\max}, \rho \}$ and $C\de (R \lambda_{\min})^{\Delta/R}/\Area(D)$.
	Then, for every $T\geq c_2 \Delta^2$, initial condition $x_0\in D$, and measurable $E\subseteq D$,
	\begin{align}
		\bbP(X_{T}\in E \mid X_0 = x_0) \geq C \Area(E).\label{eq:FairMan}
	\end{align}
\end{corollary}
\begin{proof}
	Recall that Corollary \ref{cor: TransportDistant} yields absolute constants $\delta_0, \gamma, c_3,c_4>0$ such that for every $\delta \leq \delta_0$, $\varepsilon \leq \gamma r_\delta$, and $y,x_0' \in \cD(\delta)$,
	\begin{equation}
		\bbP(X_{T_1} \in \sB(y,\varepsilon) \mid X_0 = x_0') \geq  (c_3 r_\delta \lambda_{\min})^{c_4\Delta/r_\delta} (\varepsilon / r_\delta )^2  \textnormal{ if }  T_1 \de (2\lceil 10 \Delta/ r_\delta \rceil + 2) r_\delta^2.
	\end{equation}
	Hence, possibly reducing $\delta_0$ further, the reduction from Lemma \ref{lem: x0_reduction_distant} yields a constant $c_5>0$ such that for every $T_2 \geq T_1 +  c_5 \min\{1/\kappa^2, 1/\lambda_{\max}^2, \rho^2 \}$, also if $x_0 \not\in \cD(\delta)$,
	\begin{align}
		\bbP(X_{T_2} \in \sB(y,\varepsilon) \mid X_0 = x_0) & \geq \frac{1}{2}\inf_{x_0' \in \cD(\delta)}  \bbP(X_{T_1} \in \sB(y,\varepsilon) \mid X_0 = x_0')\nonumber \\
		                                                    & \geq \frac{1}{2}(c_3 r_\delta \lambda_{\min})^{c_4\Delta/r_\delta} (\varepsilon / r_\delta )^2.
	\end{align}
	Since $\Area(\sB(y,\varepsilon)) = \pi \varepsilon^2$, this ensures that the condition \eqref{eq:BoldCamel} of Corollary \ref{cor: yreduction} is satisfied.
	Hence, there exists $\delta \leq \delta_0$ and a constant $c_6>0$ such that for every measurable $E\subseteq D$,
	\begin{align}
		\bbP(X_{T_3}\in E \mid X_0 = x_0) \geq \frac{1}{4\pi}\frac{1}{r_\delta^2}(c_3 r_\delta \lambda_{\min})^{c_4\Delta/r_\delta} \Area(E) \label{eq:KindRam}
	\end{align}
	for $T_3 = T_2 +  c_6\min\{1/\kappa^2, 1/\lambda_{\max}^2, \rho^2 \}$.
	Recalling the definition of $T_1,T_2$, this shows that \eqref{eq:FairMan} holds if the following constraints are satisfied by $T$ and $C$:
	\begin{align}
		T & \geq (2\lceil 10 \Delta/ r_\delta \rceil + 2) r_\delta^2 +  (c_5 +c_6)\min\{1/\kappa^2, 1/\lambda_{\max}^2, \rho^2 \}, \label{eq:DampRam} \\
		C & \leq \frac{1}{4\pi}\frac{1}{r_\delta^2}(c_3 r_\delta \lambda_{\min})^{c_4\Delta/r_\delta}. \label{eq:JustFox}
	\end{align}
	It remains to show that these constraints are satisfied if we choose $c_1$ and $c_2$ appropriately in Corollary \ref{cor: DoeblinHolds}.
	To this end, we next consider some preparatory estimates on the quantities occurring in \eqref{eq:DampRam} and \eqref{eq:JustFox}.

	Recall the definition of $r_\delta$ from \eqref{eq: Def_R_eps} and that $\delta$ was fixed preceding \eqref{eq:KindRam}.
	Hence, for a suitable absolute constant $c_7>0$, we here have,
	\begin{align}
		r_\delta = c_7 \min\{1/\kappa, 1/\lambda_{\max}, \rho  \}.\label{eq:SafeRam}
	\end{align}
	Recall that $B_0 = \partial D$ and that the curvature of $B_0$ is bounded by $\kappa$.
	This implies that $\Area(D) \geq \pi/\kappa^2$; see \eg \cite{pankrashkin2015inequality}.
	In particular, there exists an absolute constant $c_8>0$ such that $
		\Area(D) \geq c_8 r_\delta^2.$
	Finally, recall from \eqref{eq: Def_Delta} that $\Delta$ is the diameter of $D$ with respect to the geodesic distance.
	In particular, $\Delta$ is greater than or equal to the diameter with respect to the Euclidean distance, so $\Area(D) \leq \pi \Delta^2$.
	Hence, for some absolute $c_9>0$,
	\begin{align}
		\Delta \geq c_9 r_\delta.\label{eq:DampBird}
	\end{align}

	We can now simplify \eqref{eq:DampRam} and \eqref{eq:JustFox}.
	Specifically, by \eqref{eq:DampBird}
	\begin{align}
		(2\lceil 10\Delta /r_\delta \rceil + 2)r_\delta^2 \leq c_{10} \Delta r_\delta  \leq c_{11} \Delta^2
	\end{align}
	for some sufficiently large constants $c_{10},c_{11}>0$.
	Further, by \eqref{eq:SafeRam} and \eqref{eq:DampBird},
	\begin{align}
		(c_5 + c_6)\min\{1/\kappa^2, 1/\lambda_{\max}^2, \rho^2\} = (c_5 + c_6)c_9^{-1} r_\delta^2 \leq c_{12} \Delta^2
	\end{align}
	for some $c_{12} >0$.
	Thus, the constraint in \eqref{eq:DampRam} is satisfied by $T = c_2 \Delta^2$ if we take $c_2$ sufficiently large.
	Similarly, using that $\Area(D) \leq \pi \Delta^2$, we have $(4\pi)^{-1} r_{\delta}^{-2} \geq c_{13}/\Area(D)$.
	Taking $c_1$ sufficiently small in the definition of $R$, we may assume that
	\begin{align}
		R\lambda_{\min} < c_3 r_\delta \lambda_{\min}\  \textnormal{ and }\ \Delta/r_{\delta} < \Delta /R.
	\end{align}
	Then, the constraint in \eqref{eq:JustFox} is satisfied by $C = c_{13} (R \lambda_{\min})^{\Delta/R}/\Area(D)$.
	Taking $c_1$ smaller, we can absorb the absolute constant factor $c_{13}$ in the factor $(R \lambda_{\min})^{\Delta/R}$, so that \eqref{eq:JustFox} is satisfied by $C = (R \lambda_{\min})^{\Delta/R}/\Area(D)$, as desired.
\end{proof}

\subsection{Proof of \texorpdfstring{Theorem \ref{thm: Main_mixing}}{Theorem}}\label{sec: ProofThm}
Recall that we will rely on a Doeblin minorization; see \eg \cite[Theorem 8]{roberts2004general}.
Specifically, the following formulation is applicable:
\begin{lemma}[Doeblin minorization]\label{lem: ReductionBall}
	Suppose that one is given some $C,T>0$ such that for every $x_0\in D$ and measurable $E\subseteq D$,
	\begin{align}
		\bbP(X_T \in E \mid X_0 = x_0) \geq C  \Area(E).\label{eq:LuckyUser}
	\end{align}
	Then,
	$
		\pi_{\min} \geq C
	$ and the mixing time satisfies
	$
		\tmix \leq T \lceil \ln(1/4)/\ln(1 - C\Area(D))\rceil.
	$
\end{lemma}
\begin{proof}
	Assumption \eqref{eq:LuckyUser} means that the transition kernel satisfies the Doeblin minorization condition with respect to the probability measure $Q(\cdot) \de \Area(\cdot)/\Area(D)$.
	It is classical, see \eg \cite[Theorem 8]{roberts2004general}, that then
	\begin{align}
		\sup_{E \subseteq D} \lvert \bbP(X_{nT} \in E \mid X_0 = x_0) - \pi(E)\rvert \leq (1 - C\Area(D))^n
	\end{align}
	for all $n\geq 1$ and $x_0\in D$.
	Take $n \de \lceil \ln(1/4)/\ln(1 - C\Area(D))\rceil$ and recall \eqref{eq: Def_tmix} to find the desired upper bound on $\tmix$.
	Note that $\pi(E) = \bbE_{\pi}[\bbP(X_T \in E \mid X_0)] \geq C \Area(E)$ for every measurable $E\subseteq D$ by \eqref{eq:LuckyUser} and recall \eqref{eq: Def_pimin} to also find that $\pi_{\min}\geq C$.
\end{proof}
The desired result is now immediate by Corollary \ref{cor: DoeblinHolds} and Lemma \ref{lem: ReductionBall}:
\begin{proof}[Proof of \texorpdfstring{Theorem \ref{thm: Main_mixing}}{Theorem}]
	The bound on $\pi_{\min}$ claimed in \eqref{eq:KindNose} is exactly what follows from using the constant $C$ from Corollary \ref{cor: DoeblinHolds} in Lemma \ref{lem: ReductionBall}.
	Regarding the mixing time, note that a first order Taylor expansion of $\ln(1-x)$ near $x=0$ and the fact that $\lceil x \rceil \leq 2x$ for $x\geq 1$ imply that there exist $c_1,c_2>0$ such that
	\begin{align}
		\Bigl\lceil \frac{\ln(1/4)}{\ln\bigl(1 -  (R \lambda_{\min})^{\Delta/R}\bigr)} \Bigr\rceil \leq  c_1(R \lambda_{\min})^{-\Delta/R}\ \textnormal{ if }\ (R \lambda_{\min})^{\Delta/R} < c_2.\label{eq:VagueBox}
	\end{align}
	By taking the constant occurring in the assumed bound on $R$ in Theorem \ref{thm: Main_mixing} small, it can be ensured that the factor $(R \lambda_{\min})^{-\Delta/R}$ is greater than a given absolute constant.
	Thus, the bound on $\tmix$ claimed in \eqref{eq:KindNose} also follows.
\end{proof}

\begin{remark}\label{rem: Adaptive}
	The global approach underlying the proof can be used in some instances to derive better estimates whenever every two points in the domain can be connected through an efficient sequence of steps which locally transport some probability mass as in Figures \ref{fig: ForcedSmash} and \ref{fig: AvoidBarrier}.
	Each step which has to be taken along the route gives a constant factor reduction in the quality of the bounds, and the number of permeability-based factors which one picks up depends on the number of times that the route crosses barriers.

	We here took steps of constant size, but one could also do this adaptively by taking bigger steps in well-behaved regions of the domain where the barriers are well-spaced or have low curvature, and taking small steps only where necessary.
	There are instances where this would significantly reduce the number of steps taken and thus improve the estimate.
	On the other hand, there are also instances where approaches using local transport to establish a Doeblin minorization will fail to get efficient estimates, even if one chooses the route and step sizes smartly.
	Examples of both instance types are visualized in Figure \ref{fig: Adaptive}.
\end{remark}

\begin{figure}[t]
	\centering
	\includegraphics[width=0.85\textwidth]{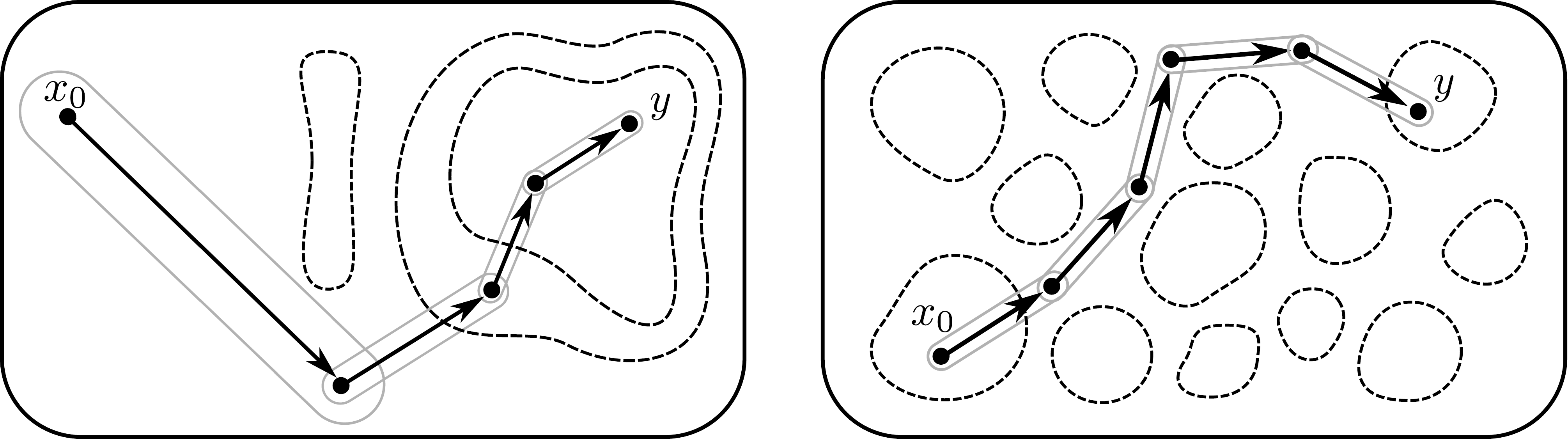}
	\caption{The proof approach underlying Theorem \ref{thm: Main_mixing} can be extended for efficient estimates avoiding worst-case behavior in some instances, but not always. (Left) Here, choosing a non-geodesic route and using location-dependent step sizes could reduce the number of local transportation steps. (Right) In a cluttered rectangular domain, a proof based on local transportation will give exponentially decaying estimates as the size of the domain increases, as one has to use many small steps. The true mixing time is however expected to behave similarly to a random walk on a square, which has order $\Delta^2$. }
	\label{fig: Adaptive}
\end{figure}

\subsection*{Acknowledgements}
This work is part of the project Clustering and Spectral Concentration in Markov Chains with project number OCENW.KLEIN.324 of the research programme Open Competition Domain Science – M, which is partly financed by the Dutch Research Council (NWO).
AVW additionally acknowledges funding by the Deutsche Forschungsgemeinschaft (DFG, German Research Foundation) under Germany’s Excellence Strategy EXC 2044–390685587, Mathematics Münster: Dynamics–Geometry–Structure.

\bibliographystyle{abbrv}

\end{document}